\documentclass[12pt]{amsart}

\usepackage[colorlinks=true]{hyperref}
\usepackage[all]{xy}
\usepackage{amssymb}
\usepackage{aliascnt}
\usepackage{verbatim}
\usepackage{tikz}
\usetikzlibrary{decorations.markings}
\usepackage{amsmath}
\usepackage{amsthm}

\newcommand{\A}{\mathbf{A}}
\newcommand{\C}{\mathbf{C}}
\newcommand{\G}{\mathbf{G}}
\newcommand{\bE}{\mathbb{E}}
\newcommand{\bL}{\mathbb{L}}

\newcommand{\bT}{\mathbb{T}}
\newcommand{\Z}{\mathbf{Z}}

\newcommand{\cA}{\mathcal{A}}

\newcommand{\cF}{\mathcal{F}}
\newcommand{\cK}{\mathcal{K}}

\newcommand{\cO}{\mathcal{O}}
\newcommand{\cT}{\mathcal{T}}

\newcommand{\B}{\mathrm{B}}

\newcommand{\ddr}{\mathrm{d}_\mathrm{dR}}
\renewcommand{\d}{\mathrm{d}}
\newcommand{\D}{\mathrm{D}}
\newcommand{\GL}{\mathrm{GL}}
\newcommand{\mH}{\mathrm{H}}

\newcommand{\K}{\mathrm{K}}

\newcommand{\Bordfr}{\mathrm{Bord}^{\mathrm{fr}}}
\newcommand{\Bordor}{\mathrm{Bord}^{\mathrm{or}}}
\newcommand{\cdga}{\mathrm{cdga}}
\newcommand{\Cob}{\mathrm{Cob}}
\newcommand{\Corr}{\mathrm{Corr}}
\newcommand{\CorrK}{\mathrm{Corr}_{/\K}}
\newcommand{\CorrO}{\mathrm{Corr}_{/\Omega^{2, cl}[n]}}
\newcommand{\CorrOO}{\mathrm{Corr}_{/\Omega^{2, cl}[2]}}
\newcommand{\dASt}{\mathrm{dASt}}
\newcommand{\dSt}{\mathrm{dSt}}
\newcommand{\LC}{\mathrm{LagrCorr}}
\newcommand{\Pic}{\mathrm{Pic}}
\newcommand{\QC}{\mathrm{QCoh}}
\newcommand{\SSet}{\mathrm{SSet}}
\newcommand{\Symp}{\mathrm{Symp}}

\DeclareMathOperator{\Ad}{Ad}

\DeclareMathOperator{\ch}{ch}

\DeclareMathOperator{\End}{End}
\DeclareMathOperator{\ev}{ev}
\DeclareMathOperator{\Fun}{Fun}
\DeclareMathOperator{\Hom}{Hom}
\DeclareMathOperator{\id}{id}
\DeclareMathOperator{\im}{im}
\DeclareMathOperator{\Ker}{Ker}
\DeclareMathOperator{\Loc}{Loc}
\DeclareMathOperator{\Map}{Map}
\DeclareMathOperator{\pt}{pt}
\DeclareMathOperator{\reg}{reg}
\DeclareMathOperator{\Spec}{Spec}
\DeclareMathOperator{\Sym}{Sym}
\DeclareMathOperator{\Tot}{Tot}

\newtheorem{thm}{Theorem}[section]
\newaliascnt{lm}{thm}
\newaliascnt{prop}{thm}
\newaliascnt{cor}{thm}
\newtheorem{lm}[lm]{Lemma}

\aliascntresetthe{lm}
\aliascntresetthe{prop}
\aliascntresetthe{cor}

\theoremstyle{definition}
\newtheorem*{defn}{Definition}

\begin{document}
\title{Quasi-Hamiltonian reduction via classical Chern--Simons theory}
\author{Pavel Safronov}
\address{The University of Texas at Austin, Mathematics Department, 2515 Speedway Stop C1200, Austin, Texas, USA 78712}
\begin{abstract}
This paper puts the theory of quasi-Hamiltonian reduction in the framework of shifted symplectic structures developed by Pantev, To\"{e}n, Vaqui\'{e} and Vezzosi. We compute the symplectic structures on mapping stacks and show how the AKSZ topological field theory defined by Calaque allows one to neatly package the constructions used in quasi-Hamiltonian reduction. Finally, we explain how a prequantization of character stacks can be obtained purely locally.
\end{abstract}
\maketitle

\setcounter{section}{-1}
\section{Introduction}
\subsection{}

This paper is an attempt to interpret computations of Alekseev, Malkin and Meinrenken \cite{AMM} in the framework of shifted symplectic structures \cite{PTVV}.

Symplectic structures appeared as natural structures one encounters on phase spaces of classical mechanical systems. Classical mechanics is a one-dimensional classical field theory and when one goes up in the dimension shifted, or derived, symplectic structures appear. That is, given an $n$-dimensional classical field theory, the phase space attached to a $d$-dimensional closed manifold carries an $(n-d-1)$-shifted symplectic structure. For instance, if $d = n-1$ one gets ordinary symplectic structures and for $d=n$, i.e. in the top dimension, one encounters $(-1)$-shifted symplectic spaces. These spaces can be more explicitly described as critical loci of action functionals.

An $n$-shifted symplectic structure on a stack $X$ is an isomorphism $\bT_X\stackrel{\sim}\rightarrow \bL_X[n]$ between the tangent complex and the shifted cotangent complex together with certain closedness conditions. Symplectic structures on stacks put severe restrictions on the geometry: for instance, a 0-shifted symplectic derived scheme is automatically smooth. Moreover, $n$-shifted symplectic structures for odd $n$ exist only on 0-dimensional stacks.

One can also make sense of Lagrangians $L\rightarrow X$: these are morphisms together with an identification $N_{L/X}\stackrel{\sim}\rightarrow \bL_L[n]$ between the normal bundle and the shifted cotangent complex; in the derived setting being a Lagrangian is an additional structure on a morphism. Given two Lagrangians $L_1, L_2\rightarrow X$ in an $n$-shifted symplectic stack, their intersection $L_1\times_X L_2$ carries a natural $(n-1)$-shifted symplectic structure; for instance, the critical locus of a function $S\colon M\rightarrow \A^1$ is an intersection of the graph of $\ddr S$ and the zero section inside of $T^*X$, hence it carries a $(-1)$-shifted symplectic structure.

\subsection{}

Ordinary symplectic reduction starts with a symplectic space $M$ with a $G$-action and a $G$-equivariant moment map $\mu\colon M\rightarrow \mathfrak{g}^*$ satisfying certain conditions. From this data one constructs the reduced space $M_{red} = [\mu^{-1}(0) / G]$, which again carries a symplectic structure. Note that we can also write the reduced space $M_{red}$ as $[M/G] \times_{[\mathfrak{g}^*/G]} [\pt/G]$.
 
We now come to the interpretation of the symplectic reduction in terms of shifted symplectic structures. It turns out $[\mathfrak{g}^*/G]$ carries a natural 1-shifted symplectic structure (for instance, coming from its identification with the shifted cotangent bundle $T^*[1]\B G$). A $G$-equivariant map $\mu\colon M\rightarrow\mathfrak{g}^*$ induces a Lagrangian map $[M/G]\rightarrow [\mathfrak{g}^*/G]$ if $M$ is symplectic and the usual moment map equations are satisfied. Then $M_{red}=[M/G]\times_{[\mathfrak{g}^*/G]} [\pt/G]$ is simply an intersection of two Lagrangians in $[\mathfrak{g}^*/G]$ and thus it possesses a symplectic structure \cite{Ca}.

\subsection{Quasi-Hamiltonian reduction}

Quasi-Hamiltonian reduction replaces moment maps $\mu\colon M\rightarrow \mathfrak{g}^*$ by maps $\mu\colon M\rightarrow G$. As before, $[G/G]$ carries a 1-shifted symplectic structure depending on a nondegenerate $G$-invariant quadratic form on $\mathfrak{g}$. A map $\mu\colon M\rightarrow G$ induces a Lagrangian morphism $[M/G]\rightarrow [G/G]$ if we have a $G$-equivariant two-form on $M$ together with certain conditions which imply that $M$ is a quasi-Hamiltonian space in the sense of \cite{AMM}. The reduced space \[M_{red}=[M/G]\times_{[G/G]}[\pt/G]\] is interpreted as an intersection of Lagrangians in $[G/G]$ as before.

The 1-shifted symplectic structure on $[G/G]$ has been studied previously. For instance, see \cite{Xu} where the adjoint action groupoid $G\times G\rightrightarrows G$ was shown to be quasi-symplectic, a notion closely related to that of a 1-shifted symplectic structure. More precisely, the classifying stack of a (quasi-)symplectic groupoid is a 1-shifted symplectic stack. The relation between the symplectic structure on $[G/G]$ and quasi-Hamiltonian reduction was also known before. One of the goals of this paper is to show that the 1-shifted symplectic structure on $[G/G]$ is transgressed from the 2-shifted symplectic structure on $\B G$.

As $[G/G]=\Map_{\dSt}(S^1_{\B}, \B G)$, one can try to find interpretations of the constructions appearing in the literature on quasi-Hamiltonian reduction in terms of the AKSZ topological field theory attached to $\B G$. Let us explain what it is. Let $\Bordor_n$ be the symmetric monoidal $(\infty, n)$-category of oriented bordisms. $\LC_n$ is the symmetric monoidal $(\infty, n)$-category which has
\begin{itemize}
\item objects: $(n-1)$-shifted symplectic stacks,
\item 1-morphisms: Lagrangian correspondences $X\leftarrow L\rightarrow Y$ between $(n-1)$-shifted symplectic stacks $X$, $Y$,
\item 2-morphisms: correspondences $L_1\leftarrow C\rightarrow L_2$ between Lagrangian correspondences, where $C\rightarrow L_1\times_{X\times Y} L_2$ is Lagrangian
\end{itemize}
and so on. Then the theorems in \cite{Ca} should give a splitting of the natural map \[\Fun^{\otimes}(\Bordor_n, \LC_n)\stackrel{\sim}\rightarrow \LC_n^\sim\] from the $\infty$-groupoid of symmetric monoidal functors $Z\colon \Bordor_n\rightarrow\LC_n$; the map is simply $Z\mapsto Z(\pt)$. In a sense, this can be viewed as an explicit proof of the cobordism hypothesis for the category of Lagrangian correspondences. Given an $(n-1)$-shifted symplectic stack $X$ and a manifold $M$ one assigns $Z_X(M)=\Map_{\dSt}(M_{\B}, X)$, the derived mapping stack from the constant stack associated to $M$, to $X$. As of now, only the 1-categorical truncations of $\LC_n$ have been constructed; these turn out to be enough for our purposes.

We are interested in the AKSZ field theory for $X=\B G$, a 2-shifted symplectic stack. The corresponding 3-dimensional topological field theory is the classical version of the Chern--Simons theory. Evaluation on a pair of pants gives a Lagrangian correspondence \[[G/G]\leftarrow [(G\times G)/G]\rightarrow [G/G]\times [G/G].\] Given two Lagrangian morphisms $L_1, L_2\rightarrow [G/G]$ we define their fusion to be an integral transform along this correspondence. For quasi-Hamiltonian spaces this coincides with the fusion procedure described in \cite{AMM}. Given a compact oriented surface $M$, let $M'$ be the same surface with a disk removed. Then $Z_{\B G}(M)=Z_{\B G}(M')\times_{Z_{\B G}(S^1)} Z_{\B G}(D)$. As $Z_{\B G}(D) = [\pt/G]$ and $Z_{\B G}(S^1) = [G/G]$, this gives a construction of $Z_{\B G}(M)=\Loc_G(M)$, the character stack of $M$, as a quasi-Hamiltonian reduction of $Z_{\B G}(M') = [G^{\times 2n}/G]$. The existence of the AKSZ field theory implies that the symplectic structure obtained from quasi-Hamiltonian reduction coincides with the one obtained by integrating the 2-shifted symplectic structure on $\B G$.

We should mention that the original application of the AKSZ field theory was for $X=\B\mathfrak{g}$: see the original paper \cite{AKSZ} and the paper \cite{GG} which discusses the symplectic structure on $\B\mathfrak{g}$ and the relation to ordinary Hamiltonian reduction. The AKSZ field theory for $\B\mathfrak{g}$ recovers only the formal completion of the trivial local system in the character stack $\Loc_G(M)$ (in physical terms, one is considering a perturbative Chern--Simons theory); our analysis of the symplectic structure on $\B G$ allows one to consider global questions (non-perturbative effects).

\subsection{Derived algebraic geometry}

The moduli space of isomorphism classes of $G$-local systems $\pi_0(\Loc_G(M))$ is not local, i.e. it does not satisfy descent: as any manifold can be covered by contractible open sets, every local section of $\pi_0(\Loc_G(M))$ is trivial, hence its sheafification is just a point. Therefore, we have to use the language of stacks for $\Loc_G(-)$ to define a local field theory.

The underived stack $t_0(\Loc_G(S^2))$ is isomorphic to $\B G$ since the 2-sphere is simply-connected. The stack $\B G$ does not admit a 0-shifted symplectic structure, so we are inevitably led to the land of derived algebraic geometry, where the derived stack $\Loc_G(S^2)$ does admit a 0-shifted symplectic structure.

Let us remind some basic definitions we will be using in the paper. The reader is invited to consult \cite{HAG} as well as the reviews \cite{Ca14}, \cite{TV1}, \cite{ToHDS}, \cite{ToDAG} for an introduction to derived algebraic geometry.

Let $k$ be a field of characteristic zero. The category of derived affine schemes is opposite to the category $\cdga^{\leq 0}$ of commutative differential graded algebras over $k$ concentrated in non-positive degrees. Derived prestacks are simply functors $\cdga^{\leq 0}\rightarrow \SSet$ to the category of simplicial sets. The category of derived stacks $\dSt$ is defined to be the full subcategory of prestacks satisfying \'{e}tale descent.

For a derived prestack $X$ we define the symmetric monoidal dg-category of quasi-coherent sheaves to be the limit
\[\QC(X) = \lim_{\Spec A\rightarrow X} A-\mathrm{mod},\]
where $A-\mathrm{mod}$ is the dg-category of cofibrant dg-modules over $A\in\cdga^{\leq 0}$. Given a quasi-coherent sheaf $\cF\in \QC(X)$ we denote by $\Gamma(X, \cF):=\Hom_{\QC(X)}(\cO_X, \cF)$, the complex of morphisms from the structure sheaf $\cO_X$ to $\cF$, the complex of global sections of $\cF$ over $X$.

By a stack in the paper we always mean a derived Artin stack locally of finite presentation. All such stacks have a perfect cotangent complex $\bL_X\in \QC(X)$. In this case the tangent complex $\bT_X\in \QC(X)$ is defined to be the dual sheaf $\bT_X=\mathcal{H}\mathrm{om}_{\QC(X)}(\bL_X, \cO_X)$. Note, that the stacks we consider in the paper are quotient stacks of derived affine schemes by an action of smooth affine group schemes. Hence, the stacks we get are derived Artin 1-stacks, which are 0-geometric.

Given a simplicial set $M$, we denote by $M_\B$ the constant stack associated to $M$. When we want to consider $M$ as a prestack we will simply write $M$ or $M_\bullet$.

Given a derived scheme $X$ with an action of an algebraic group $G$ we denote by $X/G$ the simplicial derived scheme which is the nerve of the action groupoid $X\times G\rightrightarrows X$. We denote by $[X/G]$ the associated derived stack. For instance, the classifying stack is
\[\B G = [\pt/G].\]

\subsection{Structure of the paper}

The paper is organized as follows. In section 1 we define the notion of shifted symplectic structures and Lagrangian morphisms. We show that one can compose Lagrangian correspondences using pullbacks, which gives rise to the composition in the category $\LC_n$. Section 2 is devoted to explicit calculations showing how ordinary symplectic reduction and quasi-Hamiltonian reduction fit into the framework of shifted symplectic structures. In section 3 we interpret the 1-shifted symplectic structure on $G$ in terms of the multiplicative $\Omega^{2, cl}$-torsor on $G$. Section 4 provides a dictionary between the operations in classical Chern-Simons theory and quasi-Hamiltonian reduction. In Section 5 we provide computations for the transgressed symplectic structures in particular showing that the symplectic structure on $[G/G]$ is the one obtained previously. We end with a discussion of prequantizations of character stacks; in particular, we show how one would obtain a prequantization of the complex-analytic character stack given a good theory of derived complex-analytic stacks.

\subsection{Acknowledgements}

The author would like to thank David Ben-Zvi for various conversations related to the content of the paper. The author would also like to thank Andrew Blumberg for explaining the basics of algebraic $\K$-theory and cyclic homology.

\section{Derived symplectic geometry}

\subsection{Symplectic structures}

Let us remind basic notions of differential forms in derived algebraic geometry \cite{PTVV}.

Let $X=\Spec A$ be a derived affine scheme for $A\in\cdga^{\leq 0}$ a non-positively graded commutative differential graded algebra.  Recall that we have the cotangent complex $\bL_A$, which is an $A$-module, and the complex of differential forms $\Omega(X):=\Sym_A (\bL_A[1])$. It has a $\G_m$ action given by scaling the cotangent complex and we denote by $\Omega^p(X)[p]$ the weight $p$ piece. Define the complex $\Omega^p(X, n)$ of $p$-forms of degree $n$ to be $\Omega^p(X)[n]$.

The space $\mH^0(\Omega^p(X, n))$ of $p$-forms of degree $n$ is an algebraic analog of the space $\mH^{p,n}(X)$ of $(p,n)$-forms in complex analytic geometry.

The de Rham differential is a morphism $\ddr\colon \Omega(X)\rightarrow \Omega(X)$ of degree $-1$ and weight $1$, which squares to zero. We define the complex of closed forms to be
\[\Omega^{cl}(X):=(\Sym_A(\bL_A[1])\otimes_k k[[u]], d + u\ddr),\]
where $d$ is the differential on $\Omega(X)$ and $u$ has degree $2$ and weight $-1$. Let $\Omega^{p,cl}(X)[p]$ be the weight $p$ piece of $\Omega^{cl}(X)$. The complex $\Omega^{p, cl}(X, n)$ of closed $p$-forms of degree $n$ is the weight $p$ piece of $\Omega^{cl}(X)[n-p]$. We have a map $\Omega^{p, cl}(X, n)\rightarrow \Omega^p(X, n)$ given by evaluation at $u=0$.

Geometrically, closed forms can be interpreted as $S^1$-equivariant functions on the free loop space \cite{TV2}, \cite{BZN}. This explains the action of $k[[u]]\cong\cO(\B S^1)$ on the complex of closed forms.

Explicitly, an element $\omega \in \mH^0(\Omega^{p, cl}(X, n))$ is a collection of differential forms $\omega_i$ for $i=0, 1, ...$, such that $\omega_i$ has weight $p+i$ and degree $n-2i$ and the following equations are satisfied:
\begin{align*}
\d\omega_0 &= 0\\
\d\omega_{i+1} + \ddr\omega_i &= 0.
\end{align*}

In other words, $\omega_{i+1}$ expresses closedness of the form $\omega_i$.

Both prestacks $\Omega^p(-, n)$ and $\Omega^{p, cl}(-, n)$ satisfy \'{e}tale descent, so we can define the complexes of forms for a general derived stack as mapping stacks
\[\Omega^p(X, n)=\Map_{\dSt}(X, \Omega^p(-, n)),\quad \Omega^{p, cl}(X, n) = \Map_{\dSt}(X, \Omega^{p, cl}(-, n)).\] For an Artin stack the complex of $p$-forms $\Omega^p(X, n)$ is the space of sections
\[\Omega^p(X, n) \cong \Gamma(X, \Sym^p_{\cO_X}(\bL_X[1])[n-p]).\]
Note that by definition $\Omega^p(-, n)$ and $\Omega^{p, cl}(-, n)$ satisfy descent as they send any colimits to limits.

A two-form $\omega\in \Omega^2(X, n)$ defines a morphism $\bT_X\rightarrow \bL_X[n]$.

\begin{defn}
A two-form $\omega\in\Omega^2(X, n)$ is \textit{nondegenerate} if $\bT_X\rightarrow \bL_X[n]$ is an isomorphism.
\end{defn}

We denote $\cA^{nd}(X, n)\subset |\Omega^2(X, n)|$ the subspace of nondegenerate forms, where $|\Omega^2(X, n)|$ is the simplicial set corresponding to the complex $\Omega^2(X, n)$ under the Dold-Kan correspondence (as the complex $\Omega^2(X, n)$ is not connective in general, we consider its truncation $\tau_{\leq 0}$). We define the space $\Symp(X, n)$ of $n$-shifted symplectic forms to be the pullback
\[
\xymatrix{
\Symp(X, n) \ar[r] \ar[d] & \cA^{nd}(X, n) \ar[d] \\
|\Omega^{2, cl}(X, n)|\ar[r] & |\Omega^2(X, n)|.
}
\]

\subsubsection{Example}

The main example of a symplectic stack relevant for this paper is the classifying stack $X=\B G$ of an affine algebraic group $G$. See \cite[Section 3.4]{TV1} for a definition of $G$-torsors over derived affine schemes. The category of quasi-coherent sheaves $\QC(\B G)$ is naturally identified with the category of comodules over $\cO(G)$. The cotangent complex of $\B G$ is $\bL_{\B G}\cong\mathfrak{g}^*[-1]\in\QC(\B G)$, where $\mathfrak{g}^*$ is the coadjoint representation of $G$. If $G$ is reductive, the functor of $G$-invariants is exact, so $\Omega^2(\B G)$ is concentrated in degree 2 and we have $\mH^0(\Omega^2(\B G, 2))\cong\Sym^2(\mathfrak{g}^*)^G$. One similarly has $\mH^0(\Omega^{2, cl}(\B G, 2))\cong \Sym^2(\mathfrak{g}^*)^G$ since $\ddr=0$. A class $\omega\in\Sym^2(\mathfrak{g}^*)^G$ is nondegenerate if the induced $G$-equivariant map $\mathfrak{g}\rightarrow\mathfrak{g}^*$ is an isomorphism.

\subsection{Lagrangian structures}

An $n$-shifted symplectic form $\omega\in \Symp(X, n)$ can be viewed as an element of $\mH^n(X, \bigwedge^2 \bL_X)$. For example, 1-shifted symplectic structures can be thought of as defining torsors over $\bigwedge^2\bL_X$ together with a trivialization of its de Rham differential and higher closedness conditions (which are void if $\bL_X$ is concentrated in nonnegative degrees). We will take up this point of view in the future sections.

Let $(X, \omega)$ be an $n$-shifted symplectic stack with $\omega\in\Omega^{2, cl}(X, n)$.

\begin{defn}
An \textit{isotropic structure} on $f\colon L\rightarrow X$ is a homotopy from $f^*\omega$ to $0$ in $\Omega^{2, cl}(L, n)$.
\end{defn}

In other words, it is an element $h\in \Omega^{2, cl}(L, n-1)$, such that $(d+u\ddr)h = f^* \omega$.

Explicitly, we have a collection of differential forms $h_i$ satisfying the conditions
\begin{align*}
\d h_0 &= f^*\omega_0 \\
\d h_{i+1} + \ddr h_i &= f^*\omega_{i+1}.
\end{align*}

The form $h_0$ defines a map $\bT_L\rightarrow \bL_L[n-1]$, which is not a chain map in general since $h_0$ is not closed. Consider instead the relative tangent bundle
\[\bT_f = f^*\bT_X[-1]\oplus \bT_L\]
with the differential given by the map $\bT_L\rightarrow f^*\bT_X$. We have a chain map $\bT_f\rightarrow\bL_L[n-1]$ defined to be $f^*\omega_0$ on the first summand and $h_0$ on the second summand.

\begin{defn}
An isotropic structure $f\colon L\rightarrow X$ is \textit{Lagrangian} if  $\bT_f\rightarrow \bL_L[n-1]$ is an isomorphism.
\end{defn}

Here is a way to unpack this definition (see \cite{Ca}). An isotropic structure on $f\colon L\rightarrow X$ is a commutativity data of the diagram
\[
\xymatrix{
\bT_L \ar[r] \ar[d] & 0 \ar[d] \\
f^*\bT_X \ar[r] & \bL_L[n]
}
\]

The isotropic structure is Lagrangian if the diagram is a pullback. In other words,
\[\bT_L\rightarrow f^*\bT_X\rightarrow \bL_L[n]\]
is an exact sequence.

\begin{thm}[\cite{PTVV}]
Let $(X, \omega)$ be an $n$-shifted symplectic stack together with two Lagrangians $L_1\rightarrow X$ and $L_2\rightarrow X$. Then their intersection $L_1\times_X L_2$ carries a natural $(n-1)$-shifted symplectic structure.
\label{thm:ptvvsympl}
\end{thm}

Let us prove a generalization of this theorem, which can also be found in \cite[Theorem 4.4]{Ca}. Let $X$ and $Y$ be $n$-shifted symplectic stacks.

\begin{defn}
A \textit{Lagrangian correspondence} is a correspondence
\[
\xymatrix{
& L \ar[dl] \ar[dr] & \\
X & & Y
}
\]
together with a Lagrangian structure on the map $L\rightarrow X\times \overline{Y}$.
\end{defn}

Here $\overline{Y}$ is $Y$ with the opposite symplectic structure.

The following theorem allows us to compose Lagrangian correspondences, which will be used in \autoref{sect:AKSZ} to describe the AKSZ topological field theory.

\begin{thm}
Let $(X,\omega_X)$, $(Y, \omega_Y)$ and $(Z, \omega_Z)$ be $n$-shifted symplectic stacks and
\[
\xymatrix{
& L_1\ar_{p_X}[dl] \ar^{p^1_Y}[dr] & & L_2\ar_{p^2_Y}[dl] \ar^{p_Z}[dr] & \\
X & & Y & & Z
}
\]
are Lagrangian correspondences. Then the pullback $L_1\times_Y L_2$ is a Lagrangian correspondence between $Z$ and $X$.
\label{thm:lagrcomp}
\end{thm}
\begin{proof}
Suppose that the Lagrangian structures on $L_1$ and $L_2$ are given by the forms $h_1$ and $h_2$ respectively, i.e.
\[p_X^*\omega_X - (p^1_Y)^*\omega_Y = (\d+u\ddr)h_1,\quad (p^2_Y)^*\omega_Y - p_Z^*\omega_Z = (\d+u\ddr)h_2.\]

Denote $L=L_1\times_Y L_2$ and let $\pi_i\colon L\rightarrow L_i$ be the projections.

Then
\[\pi_1^*p_X^*\omega_X - \pi_1^*(p^1_Y)^*\omega_Y + \pi_2^*(p^2_Y)^*\omega_Y - \pi_2^*p_Z^*\omega_Z = (\d+u\ddr)\pi_1^*h_1 + (\d+u\ddr)\pi_2^*h_2.\]

Therefore,
\[\pi_1^*p_X^*\omega_X - \pi_2^*p_Z^*\omega_Z = (\d+u\ddr)(\pi_1^*h_1 + \pi_2^*h_2),\]
i.e. $\pi_1^*h_1 + \pi_2^*h_2$ defines an isotropic structure on $L\rightarrow X\times\overline{Z}$. Let us check that it is in fact Lagrangian.

$L_1\rightarrow X\times\overline{Y}$ and $L_2\rightarrow Y\times\overline{Z}$ are Lagrangian, so we have the following pullback squares
\[
\xymatrix{
\bT_{L_1} \ar[r] \ar[d] & 0 \ar[d]\\
p_X^*\bT_X\oplus (p^1_Y)^*\bT_Y \ar[r] & \bL_{L_1}[n]
}\qquad\qquad\qquad
\xymatrix{
\bT_{L_2} \ar[r] \ar[d] & 0\ar[d] \\
(p^2_Y)^*\bT_Y \oplus p_Z^*\bT_Z \ar[r] & \bL_{L_2}[n]
}
\]

Pulling them back to $L$ and adding together we get a pullback square
\[
\xymatrix{
\pi_1^*\bT_{L_1} \oplus \pi_2^*\bT_{L_2} \ar[r] \ar[d] & 0 \ar[d] \\
\pi_1^*p_X^*\bT_X\oplus \pi_1^*(p^1_Y)^*\bT_Y \oplus \pi_2^*(p^2_Y)^*\bT_Y \oplus \pi_2^*p_Z^*\bT_Z \ar[r] & \pi_1^*\bL_{L_1}[n] \oplus \pi_2^*\bL_{L_2}[n]
}
\]

We can split off two summands of $\pi_1^*(p^1_Y)^*\bT_Y$ into the diagonal and antidiagonal parts obtaining the pullback of the form
\[
\xymatrix{
\pi^*\bT_L \oplus g^*\bT_Z \oplus \pi_1^*p_1^*\bT_Y[-1] \ar[r] \ar[d] & 0 \ar[d] \\
\pi_1^*p_X^*\bT_X \oplus \pi_2^*p_Z^*\bT_Z \ar[r] & \pi_1^*\bL_{L_1}[n] \oplus \pi_2^*\bL_{L_2}[n] \oplus \pi_1^*p_1^*\bT_Y[1]
}
\]
with the obvious differentials in the top-left and bottom-right corners. Finally, using the identification in the bottom-right corner $\bT_Y\cong \bL_Y[n]$ given by the symplectic form $\omega_Y$ we get a pullback
\[
\xymatrix{
\bT_L \ar[r] \ar[d] & 0 \ar[d] \\
\pi_1^*p_X^*\bT_X \oplus \pi_2^*p_Z^*\bT_Z  \ar[r] & \bL_L[n]
}
\]

In other words, $L\rightarrow X\times\overline{Z}$ is Lagrangian as claimed.
\end{proof}

In the case $X=Z=\pt$ \autoref{thm:lagrcomp} reduces to \autoref{thm:ptvvsympl}. Indeed, a Lagrangian $L\rightarrow \pt$ is the same as an $(n-1)$-shifted symplectic stack.

More generally, suppose $f\colon X\rightarrow Y$ is a symplectic morphism. Then the graph \[\Gamma_f=X\times_Y Y\rightarrow X\times\overline{Y}\] carries an isotropic structure. We say that the morphism $f$ is \textit{nondegenerate} if its graph is Lagrangian. In this case we can pullback Lagrangians: let $L_2\rightarrow Y$ be a Lagrangian and $L_1=\Gamma_f$. Then the theorem gives a natural Lagrangian structure on the pullback $L_2\times_Y X$. One can view \autoref{thm:lagrcomp} for $Z$ being a point as a way to perform integral transforms for Lagrangians.

\section{Symplectic reduction}

In this and future sections $G$ will denote a reductive group of finite type over $k$.

\subsection{General definition}

The general procedure for a symplectic reduction starts with a 1-shifted symplectic stack $X$ together with a choice of a Lagrangian $L\rightarrow X$. Then the data of a symplectic reduction consists of:
\begin{enumerate}
\item A stack $M$ with a $G$-action.

\item A moment map $\mu\colon [M/G]\rightarrow X$ together with a Lagrangian structure.
\end{enumerate}

The isotropic conditions $dh_0 = f^*\omega_0, dh_1 + \ddr h_0 = f^*\omega_1, ...$ will be called the \textit{moment map equations}. We will see that these equations coincide with the usual moment map equations familiar from the theory of symplectic reduction.

By definition the \textit{reduced space} is $[M/G]\times_X L$. \autoref{thm:ptvvsympl} gives a natural symplectic structure on the reduced space.

\subsection{Ordinary Hamiltonian reduction}
Let $X=[\mathfrak{g}^*/G]$. The category of quasi-coherent sheaves $\QC(X)$ is the category of $G$-equivariant sheaves on $\mathfrak{g}^*$. The tangent complex $\bT_X\in \QC(X)$ is
\[\bT_X=\mathfrak{g}\otimes_k \cO_{\mathfrak{g}^*}[1]\oplus \mathfrak{g}^*\otimes_k \cO_{\mathfrak{g}^*},\]
with the differential given by the coadjoint action.

On $\mathfrak{g}^*$ we have a canonical ``Maurer--Cartan'' form $\omega_0\in\Omega^1(\mathfrak{g}^*)\otimes_k \mathfrak{g}^*$ given by the identity map $T_x\mathfrak{g}^*=\mathfrak{g}^*\rightarrow\mathfrak{g}^*$. It defines a two-form $\omega_0\in\Omega^2(\mathfrak{g}^*/G, 1)$ of degree 1. It is closed: $(\d+u\ddr)\omega_0=0$, where $\ddr\omega_0=0$ follows from the fact that $\omega_0$ does not depend on the point $x\in\mathfrak{g}^*$ and $\d\omega_0=0$ follows from the equivariance of $\omega_0$ with respect to the coadjoint action.

It gives a morphism $\omega_0\colon \bT_X\rightarrow \bL_X[1]$
\[
\xymatrix{
\mathfrak{g}\otimes_k \cO_{\mathfrak{g}^*} \ar[r] \ar^{\sim}[d] & \mathfrak{g}^*\otimes_k \cO_{\mathfrak{g}^*} \ar[d]^{\sim} \\
\mathfrak{g}\otimes_k \cO_{\mathfrak{g}^*} \ar[r] & \mathfrak{g}^*\otimes_k \cO_{\mathfrak{g}^*}
}
\]
which is clearly an isomorphism.

We have $[\mathfrak{g}^*/G]=T^*[1]\B G$. Therefore, the symplectic structure can alternatively be defined using a canonical one-form $\theta$ of degree 1. It is given by the identity function $\id\in\cO_{\mathfrak{g}^*}\otimes_k \mathfrak{g}^*$. Then $\omega_0 = \ddr\theta$.

An isotropic structure on $[M/G]\rightarrow X$ is a closed two-form $h$ of degree 0 on $M$, which is $G$-equivariant. Moreover, the condition $\d h_0 = f^*\omega_0$ is equivalent to
\[-\iota_{a(v)} h_0 = \ddr\mu(v)\]
for $v\in\mathfrak{g}$ and $a\colon \mathfrak{g}\rightarrow\Gamma(M, \bT_M)$ the action map. Lagrangian condition translates into the fact that $h_0$ has to be nondegenerate.

For example, the map $L=[\pt/G]\rightarrow[\mathfrak{g}^*/G]$ induced from the inclusion of the origin is Lagrangian.

\subsubsection{Example}

Let $M$ be a stack with a $G$ action. We define a moment map $\mu\colon T^*M\rightarrow \mathfrak{g}^*$ as follows. The action map $a\colon \mathfrak{g}\rightarrow\Gamma(M, \bT_M)$ gives an element of \[\mathfrak{g}^*\otimes_k \Gamma(M, \bT_M)\subset \mathfrak{g}^*\otimes_k \Gamma(M, \Sym_{\cO_M} \bT_M) \cong \Gamma(T^*M, \mathfrak{g}^*\otimes_k \cO_{T^*M}).\]

Recall that the canonical one-form $\theta$ on $T^*M$ is defined to be the composite
\[\bT_{T^*M}\rightarrow p^* \bT_M\rightarrow \cO_{T^*M},\]
where $p\colon T^*M\rightarrow M$ is the projection and $p^* \bT_M\rightarrow \cO_{T^*M}$ is adjoint to \[\bT_M\rightarrow p_*\cO_{T^*M} \cong \Sym_{\cO_M} \bT_M.\] Observe that
\[\iota_{a(v)} \theta = \mu(v).\]

We define $h_0 = \ddr \theta$. The moment map equation follows from the following calculation:
\[-\iota_{a(v)} \ddr \theta = \ddr\iota_{a(v)} \theta = \ddr\mu(v),\]
where we used $G$-invariance of $\theta$ in the second equality.

The symplectic reduction \[[T^*M/G]\times_{[\mathfrak{g}^*/G]} [\pt/G]\] is isomorphic to $T^*[M/G]$.

\subsection{Quasi-Hamiltonian reduction}

Consider the right action of $G$ on itself by conjugation: $a\mapsto g^{-1}ag=:\Ad_g(a)$ and let $X=[G/G]$. The tangent complex is
\[\mathfrak{g}\otimes_k\cO_G\rightarrow \bT_G\]
in degrees $-1$ and 0 with the differential $\mathfrak{g}\rightarrow \Gamma(G, \bT_G)$ given by the adjoint action \[x\in\mathfrak{g}\mapsto x^R - x^L,\] where $x^L$ and $x^R$ are vector fields generating the left and right actions of $G$ on itself. The cotangent complex is
\[\bL_G\rightarrow \mathfrak{g}^*\otimes_k \cO_G\]
in degrees 0 and 1 with the differential $\d$ given by
\[(\d\phi)(x) = -\iota_{(x^R-x^L)}\phi\]
for $\phi\in\bL_G$. At any point $a\in G$ we have $x^L = \Ad_a x^R$.

Recall the left and right Maurer--Cartan forms $\theta,\overline{\theta}\in\Omega^1(G)\otimes_k\mathfrak{g}$ defined by
\[\iota_v\theta = (a\in G\mapsto (L_{a^{-1}})_*v_a),\quad \iota_v\overline{\theta} = (a\in G\mapsto (R_{a^{-1}})_* v_a)\]
for a vector field $v\in\Gamma(G, \bT_G)$. For any point $a\in G$ we have
\[\theta = \Ad_a\overline{\theta}.\] The contraction of the Maurer--Cartan forms with the invariant vector fields are as follows:
\[\iota_{(x^L)}\theta = \Ad_a(x), \iota_{(x^R)}\theta = x, \iota_{(x^L)}\overline{\theta} = x, \iota_{(x^R)}\overline{\theta} = \Ad_{a^{-1}}(x).\]
Furthermore, we have the Maurer--Cartan equations
\[\ddr\theta +\frac{1}{2}[\theta, \theta] = 0,\quad \ddr\overline{\theta} - \frac{1}{2}[\overline{\theta}, \overline{\theta}] = 0.\]

The sheaf of two-forms on $[G/G]$ is \[\bigwedge^2 \bL_G \oplus \bL_G\otimes_k \mathfrak{g}^*[-1] \oplus \cO_G\otimes_k\Sym^2(\mathfrak{g}^*)[-2].\] Let $(-,-)\colon \mathfrak{g}\otimes_k\mathfrak{g}\rightarrow k$ be a $G$-invariant nondegenerate symmetric bilinear form. Then we can define a two-form $\omega_0$ of degree 1 by
\begin{equation}
\omega_0(y) = -\frac{1}{2}(\theta+\overline{\theta}, y)
\end{equation}
for any $y\in\mathfrak{g}$.

\begin{lm}
$\omega_0$ is $\d$-closed.
\end{lm}
\begin{proof}
If we view $\d\omega_0$ as an element of $\mathfrak{g}^*\otimes_k\mathfrak{g}^*$, we have to prove that it is antisymmetric.

\begin{align*}
\d\omega_0(x, y) &= \frac{1}{2}(\iota_{(x^R-x^L)}\theta + \iota_{(x^R-x^L)}\overline{\theta}, y) \\
&= \frac{1}{2}(x - \Ad_a(x) + \Ad_{a^{-1}}(x) - x, y) \\
&= \frac{1}{2}(\Ad_{a^{-1}}(x), y) - \frac{1}{2}(x, \Ad_{a^{-1}}(y)).
\end{align*}
\end{proof}

Although $\omega_0$ is not $\ddr$-closed, it is homotopically $\ddr$-closed: there is a differential form $\omega_1$, such that $\ddr\omega_0 + d\omega_1 = 0$. Indeed, define a three-form $\omega_1$ of degree 0 by
\begin{equation}
\omega_1 = \frac{1}{12}(\theta, [\theta, \theta]).
\end{equation}

\begin{lm}
The equation $\ddr\omega_0 + \d\omega_1 = 0$ is satisfied.
\end{lm}
\begin{proof}
For $x\in\mathfrak{g}$ we must prove
\[-\ddr\frac{1}{2}(\theta + \overline{\theta}, x) - \iota_{(x^R-x^L)}\omega_1 = 0.\]

Let us split
\[\omega_1 = \frac{1}{24}(\theta, [\theta, \theta]) + \frac{1}{24}(\overline{\theta}, [\overline{\theta}, \overline{\theta}]).\] Then we have to prove
\[\frac{1}{4}([\theta, \theta], x) - \frac{1}{4}([\overline{\theta}, \overline{\theta}], x) - \frac{1}{8}(\iota_{(x^R-x^L)}\theta, [\theta, \theta]) - \frac{1}{8} (\iota_{(x^R-x^L)}\overline{\theta}, [\overline{\theta}, \overline{\theta}])=0.\]

This is equivalent to
\[2([\theta, \theta], x) - 2([\overline{\theta}, \overline{\theta}], x) - (x - \Ad_a(x), [\theta, \theta]) - (\Ad_{a^{-1}}(x)-x, [\overline{\theta}, \overline{\theta}])=0.\]

The claim follows from the invariance of the bilinear form under conjugation.
\end{proof}

\begin{lm}
The form $\omega_1$ is $\ddr$-closed.
\end{lm}
\begin{proof}
From the Maurer--Cartan equation we see that $[\theta, \theta]$ is $\ddr$-closed. Then
\[\ddr\omega_1 = \frac{1}{12}(\ddr\theta, [\theta, \theta]) = -\frac{1}{12}([\theta, \ddr\theta], \theta) = \frac{1}{12}(\ddr[\theta, \theta], \theta) = 0,\]
where we used invariance of the bilinear form in the second equality.
\end{proof}

The previous three lemmas prove that $\omega_0 + u\omega_1$ is a closed two-form. To see that it is symplectic, we have to check that it is nondegenerate.

\begin{lm}
The two-form $\omega_0\colon \bT_{[G/G]}\rightarrow \bL_{[G/G]}[1]$ is nondegenerate.
\end{lm}
\begin{proof}
$\omega_0$ gives the following chain map:
\[
\xymatrix{
\mathfrak{g}\otimes_k \cO_G \ar[r]\ar[d] & \bT_G\ar[d] \\
\bL_G \ar[r] & \mathfrak{g}^*\otimes_k \cO_G,
}
\]
where the vertical maps are dual to each other. As the vertical maps are morphisms of vector bundles of the same rank, we just have to check that one of them (say, the left one) is injective on cohomology.

Consider a point $a\in G$ and a closed degree 0 element $v\in\mathfrak{g}$ of $\bT_{[G/G]}[-1]$. Closedness of $v$ is equivalent to the equation $\Ad_{a^{-1}}v = v$.

Its image under $\omega_0$ is
\[-\frac{1}{2}(\theta + \overline{\theta}, v).\]

If this form is zero, its contraction with every vector field of the form $x^L$ is zero as well. That is,
\[-\frac{1}{2}(\Ad_a(x) + x, v) = 0.\]

However,
\[-\frac{1}{2}(\Ad_a(x) + x, v) = -\frac{1}{2}(x, \Ad_{a^{-1}}(v) + v) = -(x, v).\]

It is zero for all $x\in\mathfrak{g}$ if and only if $v = 0$, i.e. the left vertical map is injective.
\end{proof}

Consider a $G$-equivariant map $\mu\colon M\rightarrow G$ of \textit{right} $G$-spaces. It induces an isotropic map $[M/G]\rightarrow [G/G]$ if we are given a two-form $h_0$ of degree 0 on $[M/G]$, such that
\[\mu^*\omega_0 = \d h_0,\quad \mu^*\omega_1 = \ddr h_0.\]
Substituting the expressions for $\omega_0$ and $\omega_1$ we get
\begin{align*}
\iota_{a(v)}h_0 = \frac{1}{2}\mu^*(\theta+\overline{\theta}, v)\\
\ddr h_0 = \frac{1}{12}\mu^*(\theta, [\theta, \theta]).
\end{align*}
These are precisely the moment map equations for the quasi-Hamiltonian reduction. One sees that the equations coincide with \cite[Definition 2.2]{AMM} up to a sign in the second equation since \cite{AMM} consider left $G$-actions. In the future we will call Lagrangians $X\rightarrow [G/G]$ \textit{quasi-Hamiltonian spaces}.

\subsubsection{Example}

This example is due to Alekseev, Malkin and Meinrenken \cite[Section 9]{AMM}.

Let $M$ be a closed oriented surface together with a point $x\in M$. Let $\Loc_G(M)$ be the moduli space of local systems on $M$ also known as the \textit{character stack}. The moment map
\[\mu\colon\Loc_G(M\backslash x)\rightarrow [G/G]\]
is given by the monodromy around the puncture $x$. This gives $\Loc_G(M\backslash x)$ the structure of a quasi-Hamiltonian space. The symplectic reduction of $\Loc_G(M\backslash x)$ is $\Loc_G(M)$, which inherits a symplectic form.

For instance, let $M = T^2$ be the 2-torus. More general character varieties can be obtained by fusion (see the next section). We have a moment map
\[\mu\colon G\times G\rightarrow G\]
given by the commutator $a, b\mapsto aba^{-1}b^{-1}$. The two-form $h_0$ on $G\times G$ is given by
\begin{equation}
h_0 = \frac{1}{2}(p_1^*\theta, p_2^*\overline{\theta}) + \frac{1}{2}(p_2^*\theta, p_1^*\overline{\theta}) + \frac{1}{2}(m^*\theta, i^*m^* \overline{\theta}),
\end{equation}
where $m\colon G\times G\rightarrow G$ is the multiplication, $p_i\colon G\times G\rightarrow G$ are the projection and $i\colon G\times G\rightarrow G\rightarrow G$ is the inversion on each factor.

We will compute the form $h_0$ in the AKSZ formalism in the last section.

\section{Multiplicative torsor on the group}

\label{sect:multiplicative}

\subsection{Multiplicative structures}

\subsubsection{}

In this section we will show that there is a multiplicative $\Omega^{2,cl}$-torsor on $G$, which gives rise to the 1-shifted symplectic form on $[G/G]$ described previously.

Let $\cA$ be a natural system of sheaves of abelian groups on stacks. That is, it is a collection of sheaves $\cA_X$ for every stack $X$ together with compatible maps $f^{-1}\cA_Y\rightarrow \cA_X$ for every morphism $f\colon X\rightarrow Y$. Given an $\cA_Y$-torsor $\cT$ on $Y$, we define the pullback $\cA$-torsor $f^*\cT$ to be
\[f^*\cT = f^{-1}\cT\times_{f^{-1}\cA_Y}\cA_X.\]

\begin{defn}
A \textit{multiplicative $\cA$-torsor} $\cT$ on $G$ is an $\cA$-torsor $\cT$ together with the data of an isomorphism
\[\phi\colon  m^*\cT\cong p_1^*\cT\times_{\cA} p_2^*\cT=:\cT\boxtimes \cT\]
satisfying the following pentagon diagram expressing associativity:
\[
\xymatrix{
m_{12}^*m^*\cT \ar_{m_{12}^*\phi}[d] \ar^{\sim}[rr] && m_{23}^* m^*\cT \ar^{m_{23}^*\phi}[d] \\
m_{12}^*(\cT\boxtimes \cT) \ar_{\phi\boxtimes\id}[dr] && m_{23}^*(\cT\boxtimes \cT) \ar^{\id\boxtimes\phi}[dl] \\
& \cT\boxtimes \cT\boxtimes \cT &
}
\]
\end{defn}

The maps $m_{12}, m_{23}\colon G\times G\times G\rightarrow G\times G$ are multiplications of the first two and the last two factors respectively.

Let $\B G$ be the classifying stack of a group $G$ and let the simplicial scheme $\B_\bullet G$ be the nerve of the map $\pt\rightarrow G$ classifying the trivial torsor. The simplicial scheme $\B_\bullet G$ is $G^n$ in degree $n$ with the face maps coming from the multiplication of adjacent elements.

Suppose that all $\cA$-gerbes on a point admit a trivialization. Then a multiplicative torsor is just an element of
\[\Tot \Gamma(\B_\bullet G, \cA[2]).\]
Indeed, an element $\cT\in \Tot \Gamma(\B_\bullet G, \cA[2])$ is an $\cA$-gerbe on a point which we assume to be trivial together with an isomorphism between two pullbacks $G\rightrightarrows \pt$ given by an $\cA$-torsor $\cT$ on $G$. Finally, we have a trivialization of $p_2^*\cT\times_\cA m^*\cT^{-1}\times_\cA p_1^*\cT$ on $G\times G$, i.e. an identification $m^*\cT\cong p_2^*\cT\times_\cA p_1^*\cT$ satisfying the associativity condition written above.

More generally, given a complex of sheaves of abelian groups $\cA$, we define a multiplicative torsor over $\cA$ to be an element of $\Tot\Gamma(\B_\bullet G, \cA[2])$.

By the universal property of totalization we have a natural map
\[\Gamma(\B G, \cA[2])\rightarrow \Tot\Gamma(\B_\bullet G, \cA[2]).\]
If $\cA$ satisfies descent with respect to the smooth topology, this map is an isomorphism. Hence, in this case a multiplicative $\cA$-torsor on $G$ is the same as an $\cA$-gerbe on $\B G$.

Given a multiplicative torsor over $\cA$, we can descend it to an $\cA$-torsor on the adjoint quotient $[G/G]$. Indeed, let $f\colon G\rightarrow G\times G$ be the map that sends $g$ to $(g^{-1}, g)$. Then $fm$ is the constant map that sends $g\mapsto e$. Therefore, we have a trivialization
\[\cA\cong f^*m^*\cT\stackrel{\phi}\cong f^*(\cT\boxtimes \cT).\]

Consider the composite $\Ad\colon  G\times G\stackrel{p_2\times f_{13}}\rightarrow G\times G\times G\stackrel{m}\rightarrow G$ given by $a,g\mapsto (g^{-1},a,g)\mapsto g^{-1}ag$. Then the pullback of $\cT$ along $\Ad$ is isomorphic to $\cT\boxtimes \cA$.

A section $s\in \mH^0(G, \cT)$ is $G$-invariant, i.e. is a pullback of a section over $[G/G]$, if the element $\Ad^*s\in \mH^0(G\times G, \Ad^* \cT)$ coincides with $p_1^* s\in \mH^0(G\times G, \cT\boxtimes \cA)$ under the isomorphism $\Ad^*\cT\cong \cT\boxtimes \cA$.

\subsection{Multiplicative torsors over $\cA=\Omega^2$}

As $G$ is affine, $\Gamma(\B_\bullet G, \Omega^2[2])$ is concentrated in degree 2. Therefore, an element \[\cT\in \mH^0(\Tot\Gamma(\B_\bullet G, \Omega^2[2]))\] boils down to a two-form $\phi\in \Omega^2(G\times G)$ satisfying the associativity condition
\begin{equation}
m_{23}^*\phi + p_{23}^*\phi = m_{12}^*\phi + p_{12}^*\phi.
\label{eq:assoc}
\end{equation}

Let us fix this form to be
\begin{equation}
\phi = -\frac{1}{2}(p_1^*\theta, p_2^*\overline{\theta}).
\end{equation}

To check associativity, let us write down the pullbacks of the Maurer--Cartan forms under multiplication:
\[m^*\theta = \Ad_b p_1^*\theta + p_2^*\theta,\quad m^*\overline{\theta} = p_1^*\overline{\theta} + \Ad_{a^{-1}}p_2^*\overline{\theta},\]
where $a$ and $b$ are coordinates on the two factors of $G$. Hence, the associativity condition becomes
\[(p_1^*\theta, p_2^*\overline{\theta} + \Ad_{b^{-1}}p_3^*\overline{\theta}) + (p_2^*\theta, p_3^*\overline{\theta}) = (\Ad_{b}p_1^*\theta + p_2^*\theta, p_3^*\overline{\theta}) + (p_1^*\theta, p_2^*\overline{\theta}).\]

\subsubsection{}
Finally, let us work out what it means for a section $s\in \mH^0(G, \cT)\cong \mH^0(G, \Omega^2)$ to be invariant under conjugation. As before, denote by $f\colon G\rightarrow G\times G$ the map $g\mapsto (g^{-1}, g)$ and $p_2\times f_{13}\colon G\times G\rightarrow G\times G\times G$ the map $(a, g)\mapsto (g^{-1}, a, g)$. The section $s$ is $G$-invariant if
\[\Ad^*s - f_{13}^*m_{12}^* \phi - f_{13}^*p_{12}^* \phi + f_{13}^*p_{13}^*\phi = p_1^*s.\]

The term $f_{13}^*p_{13}^*\phi$ vanishes, since it is equal to $\frac{1}{2}(p_2^*\overline{\theta}, p_2^*\overline{\theta})=0$. The other two terms containing $\phi$ are
\begin{align*}
\frac{1}{2}f_{13}^*(\Ad_gp_1^*\theta + p_2^*\theta, p_3^*\overline{\theta}) + \frac{1}{2}f_{13}^*(p_1^*\theta, p_2^*\overline{\theta}) &= \frac{1}{2}(-\Ad_ap_2^*\overline{\theta} + p_1^*\theta, p_2^*\overline{\theta}) - \frac{1}{2}(p_2^*\overline{\theta}, p_1^*\overline{\theta}) \\
&= \frac{1}{2}(p_1^*\theta - \Ad_ap_2^*\overline{\theta} + p_1^*\overline{\theta}, p_2^*\overline{\theta}).
\end{align*}

Therefore, a section $s$ is $G$-invariant if
\begin{equation}
\Ad^*s = p_1^*s - \frac{1}{2}(p_1^*\theta - \Ad_ap_2^*\overline{\theta} + p_1^*\overline{\theta}, p_2^*\overline{\theta}).
\label{invsectioncech}
\end{equation}

Picking out different components of this equation, we get the following consequences:
\begin{enumerate}
\item Restricting to $G\times \{g\}\subset G\times G$, we get
\[\Ad_g^* s = s,\]
i.e. the form $s$ has to be invariant under the adjoint action.

\item Contracting the equation with a vector field $v^L$ generating a left action on the second factor of $G$ and then restricting it to the first factor of $G$, we get
\begin{equation}
\iota_{(v^R - v^L)}s = \frac{1}{2}(\theta+\overline{\theta}, v).
\label{invsection}
\end{equation}

\item Finally, contracting the equation with $v^L$ and $w^L$ along the second $G$ factor, we get
\[\iota_{(w^R - w^L)} \iota_{(v^R - v^L)} s = \frac{1}{2}(\Ad_a (v), w) - \frac{1}{2}(\Ad_a(w), v) = \frac{1}{2}(\Ad_{a^{-1}}(w) - \Ad_a(w), v),\]
which follows from equation \eqref{invsection} since $\iota_{(w^R - w^L)}(\theta+\overline{\theta}) = \Ad_{a^{-1}}(w) - \Ad_a(w)$.
\end{enumerate}

On $[G/G]$ we can write the equation \eqref{invsection} as
\[\d s(x) = -\frac{1}{2}(\theta+\overline{\theta}, x) = \omega_0(x).\]
In other words, $s$ is a section of the $\Omega^2$-torsor with class $\omega_0\in \mH^1([G/G], \Omega^2)$.

\subsection{Multiplicative torsors over $\cA=\Omega^{2, cl}$}

Since $G$ is affine, the fibers of the forgetful map
\[\{\text{multiplicative $\Omega^{2,cl}$-torsors $\cT$}\}\rightarrow \{\text{multiplicative $\Omega^2$-torsors $\cT$}\}\]
consist of multiplicative sections of the induced $\Omega^3$-torsor $\ddr\cT$. Explicitly, these are 3-forms $s$, such that
\[m^*s + \ddr\phi = p_1^*s + p_2^*s.\]

\begin{lm}
The three-form $\omega_1=\frac{1}{12}(\theta, [\theta, \theta])$ is a multiplicative section of the multiplicative $\Omega^2$-torsor defined by the two-form $-\phi$.
\label{lemma:chimult}
\end{lm}
\begin{proof}
Recall that $m^*\theta = \Ad_bp_1^*\theta + p_2^*\theta$. Therefore,
\begin{align*}
m^*\omega_1 &= \frac{1}{12}(m^*\theta, [m^*\theta, m^*\theta]) \\
&= \frac{1}{12}(\Ad_bp_1^*\theta + p_2^*\theta, [\Ad_bp_1^*\theta + p_2^*\theta, \Ad_bp_1^*\theta + p_2^*\theta]).
\end{align*}

We also have
\begin{align*}
\ddr\phi &= -\frac{1}{2}(p_1^*\ddr\theta, p_2^*\overline{\theta}) + \frac{1}{2}(p_1^*\theta, p_2^*\ddr\overline{\theta}) \\
&= \frac{1}{4}(p_1^*[\theta, \theta], p_2^*\overline{\theta}) + \frac{1}{4}(p_1^*\theta, p_2^*[\overline{\theta}, \overline{\theta}])
\end{align*}

There are eight terms in $m^*\omega_1$. Two of them are just $p_1^*\omega_1 + p_2^*\omega_1$. Another six terms break into two triples:
\[
\frac{1}{12}(\Ad_bp_1^*\theta, [p_2^*\theta, p_2^*\theta]) + \frac{1}{12}(p_2^*\theta, [\Ad_b p_1^*\theta, p_2^*\theta]) +\frac{1}{12}(p_2^*\theta, [p_2^*\theta, \Ad_b p_1^*\theta]) = \frac{1}{4}(p_1^*\theta, [p_2^*\overline{\theta}, p_2^*\overline{\theta}])
\]
and similarly for the other triple. We see that these six terms cancel with the terms in $\ddr\phi$.
\end{proof}

To summarize, we have constructed a multiplicative torsor over $\Omega^{2, cl}$ on $G$, such that the induced $\Omega^{2, cl}$ torsor on $[G/G]$ is represented by the differential forms $(\omega_0, \omega_1)$.

\section{AKSZ topological field theory}

\label{sect:AKSZ}

\subsection{Lagrangian correspondences}

Consider the symmetric monoidal 1-category $\LC_n$ whose objects are $(n-1)$-shifted symplectic stacks and morphisms $X\rightarrow Y$ are Lagrangian correspondences $X\leftarrow L\rightarrow Y$. \autoref{thm:lagrcomp} defines a composition on this category. To make the composition well-defined, we consider Lagrangian correspondences only up to an isomorphism. Let us spell out the notion of isomorphisms explicitly.

Two Lagrangians $f_i\colon L_i\rightarrow X\times \overline{Y}$ are isomorphic if we have an isomorphism of stacks $g\colon L_1\rightarrow L_2$ together with a commutative diagram
\[
\xymatrix{
L_1 \ar^{f_1}[r] \ar_{g}[d] & X\times\overline{Y} \\
L_2 \ar_{f_2}[ur] &
}
\]

We have a loop at the origin in $\Omega^{2, cl}(L_2, n)$ given by the path \[0\leadsto (f_2^*\omega_X - f_2^*\omega_Y) \leadsto (g^*f_1^*\omega_X - g^*f_1^*\omega_Y) \leadsto 0\] which we require to be contractible.

The symmetric monoidal structure on $\LC_n$ is given by the Cartesian product of symplectic stacks.

Recall also the symmetric monoidal 1-category of cobordisms $\Cob^{or}_n$ whose objects are closed oriented $(n-1)$-manifolds and morphisms are oriented cobordisms between them.

Given a topological space $M$ we can assign to it a constant stack $M_{\B}$. Let us recall the following two theorems (\cite[Theorem 2.5]{PTVV} and \cite[Section 3.1.2]{Ca}).

\begin{thm}
Let $M$ be a closed oriented $(n-1)$-manifold and $X$ an $m$-shifted symplectic stack. Then the derived mapping stack $\Map_{\dSt}(M_{\B}, X)$ carries a natural $(m-n+1)$-shifted symplectic structure.
\label{thm:PTVVAKSZ}
\end{thm}

\begin{thm}
Let $M$ be a compact oriented $n$-manifold. Then the restriction map \[\Map_{\dSt}(M_{\B}, X)\rightarrow \Map_{\dSt}((\partial M)_{\B}, X)\] carries a natural Lagrangian structure.
\label{thm:CalAKSZ}
\end{thm}

One can recover the previous theorem since $\partial M\cong\varnothing$ and Lagrangian maps from the stack $\Map_{\dSt}(M_{\B}, X)$ into the point equipped with a unique $(m-n+1)$-shifted symplectic structure are the same as $(m-n)$-shifted symplectic structures on $\Map_{\dSt}(M_{\B}, X)$.

Following \cite{Ca} we define the AKSZ topological field theory $Z_X\colon \Cob^{or}_n\rightarrow \LC_{m-n+2}$ whose value on any manifold $M$ is given by the derived mapping stack \[Z_X(M)=\Map_{\dSt}(M_{\B}, X).\] See \textit{loc. cit} for more details.

\subsection{Classical Chern--Simons theory}

We would like to interpret objects appearing in \cite{AMM} from the point of view of the AKSZ topological field theory.

The classifying stack $\B G$ carries a 2-shifted symplectic structure constructed from a nondegenerate $G$-invariant quadratic form $q\in\Sym^2(\mathfrak{g}^*)^G$. The field theory \[Z_{\B G}\colon \Cob^{or}_2\rightarrow \LC_2\] is the classical Chern--Simons theory truncated to dimensions 1 and 2. Let's consider some simple cobordisms.

\begin{itemize}
\item $M=S^1$. $Z_{\B G}(S^1) = [G/G]$ and it carries a 1-shifted symplectic structure.

\item $M$ is the disk. Then $Z_{\B G}(M)=[\pt/G]$ which carries a Lagrangian map \[[\pt/G]\rightarrow [G/G]\] given by the inclusion of the identity element.

\item $M=S^1\times I$ viewed as a cobordism from $\pt$ to $S^1\sqcup S^1$. We call \[D(G):=Z_{\B G}(S^1\times I)=[G/G]\] the double of $G$. The map $[G/G]\rightarrow [G/G]\times [G/G]$ given by $a\mapsto (a, a^{-1})$ is Lagrangian.

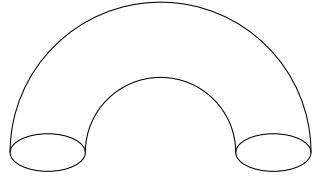
\begin{figure}[h]
\centering
\begin{tikzpicture}
  \draw (0,0) arc (0:180:2);
  \draw (-0.5,0) ellipse (0.5 and 0.25);
  \draw (-3.5, 0) ellipse (0.5 and 0.25);
  \draw (-1,0) arc (0:180:1);
\end{tikzpicture}
\caption{The double $D(G)$.}
\end{figure}

\item $M$ is a pair of pants viewed as a cobordism from $S^1\sqcup S^1$ to $S^1$. Then \[Z_{\B G}(M) = [(G\times G)/G].\] The AKSZ field theory then gives a Lagrangian correspondence
\[
\xymatrix{
& [(G\times G)/G] \ar[dl] \ar[dr] & \\
[G/G] & & [G/G] \times [G/G]
}
\]

\begin{figure}[h]
\centering
\begin{tikzpicture}
  \draw (0.05,0.1) .. controls (1, 1) .. (1, 2);
  \draw (0.4,0) ellipse (0.4 and 0.2);
  \draw (1.4,2) ellipse (0.4 and 0.2);
  \draw (1.8, 2) .. controls (1.8, 1) .. (2.75, 0.1);
  \draw (2.4, 0) ellipse (0.4 and 0.2);
  \draw (0.75, -0.1) .. controls (1.2, 0.4) and (1.6, 0.4) .. (2.05, -0.1);
\end{tikzpicture}
\caption{Fusion}
\end{figure}

For example, for any Lagrangian $L\rightarrow [G/G]\times [G/G]$ we get a Lagrangian \[L\times_{[(G\times G)/G]}([G/G]\times [G/G])\rightarrow [G/G]\] which is called the \textit{internal fusion} of $L$.

\item $M$ is a 2-torus with a disk removed. We can view $M$ as a composition of the cylinder with a pair of pants, so $Z_{\B G}(M)$ is the fusion of the double $D(G)$. Explicitly, \[Z_{\B G}(M)=[(G\times G)/G]\] with a Lagrangian morphism $Z_{\B G}(M)\rightarrow [G/G]$ given by $(a, b)\mapsto aba^{-1}b^{-1}$.

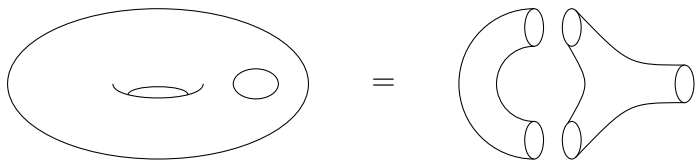
\begin{figure}[h]
\centering
\begin{tikzpicture}
  \draw (0,0) ellipse (2 and 1);
  \draw (-0.6, 0) .. controls (-0.6, -0.25) and (0.6, -0.25) .. (0.6, 0);
  \draw (-0.4, -0.15) .. controls (-0.4, 0) and (0.4, 0) .. (0.4, -0.15);
  \draw (1.3, 0) ellipse (0.3 and 0.2);
  
  \draw (3, 0) node {=};

  \draw (5, 1) arc (90:270:1);
  \draw (5, 0.75) ellipse (0.125 and 0.25);
  \draw (5, -0.75) ellipse (0.125 and 0.25);
  \draw (5, 0.5) arc (90:270:0.5);

  \draw (5.5, 0.75) ellipse (0.125 and 0.25);
  \draw (5.5, -0.75) ellipse (0.125 and 0.25);
  \draw (5.55, 0.98) .. controls (6.25, 0.25) .. (7, 0.25);
  \draw (5.55, -0.98) .. controls (6.25, -0.25) .. (7, -0.25);
  \draw (7, 0) ellipse (0.125 and 0.25);
  \draw (5.45, 0.52) .. controls (5.75, 0) .. (5.45, -0.52);
\end{tikzpicture}
\caption{Punctured torus as a fusion of $D(G)$}
\end{figure}

\item $M$ is a closed oriented surface of genus $g$. We can split it as a composition of $M$ with a disk removed $M'$ and a disk. This allows us to compute
\[Z_{\B G}(M) = Z_{\B G}(M')\times_{[G/G]} [\pt/G].\]

In other words, the character variety of $M$ with its symplectic structure obtained by the AKSZ construction can be also obtained from a quasi-Hamiltonian reduction of \[Z_{\B G}(M') = [(\underbrace{G\times ... \times G}_{\text{$2g$ times}})/G].\]

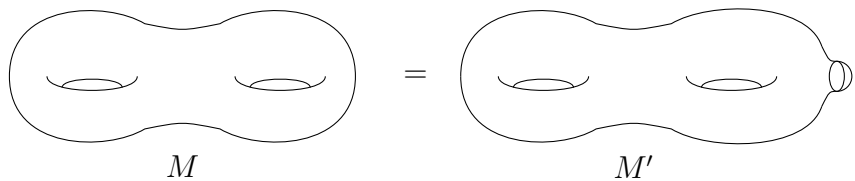
\begin{figure}[h]
\centering
\begin{tikzpicture}
  \draw (0,0) .. controls (0,1) and (1.3, 1) .. (1.8, 0.7)
  .. controls (2.3, 0.6) .. (2.8, 0.7)
  .. controls (3.3, 1) and (4.6, 1) .. (4.6, 0)
  .. controls (4.6, -1) and (3.3, -1) .. (2.8, -0.7)
  .. controls (2.3, -0.6) .. (1.8, -0.7)
  .. controls (1.3, -1) and (0, -1) .. (0, 0);
  \draw (0.5, 0) .. controls (0.5, -0.25) and (1.7, -0.25) .. (1.7, 0);
  \draw (0.7, -0.15) .. controls (0.7, 0) and (1.5, 0) .. (1.5, -0.15);
  \draw (3, 0) .. controls (3, -0.25) and (4.2, -0.25) .. (4.2, 0);
  \draw (3.2, -0.15) .. controls (3.2, 0) and (4, 0) .. (4, -0.15);
  \draw (2.3, -1.2) node {$M$};
  
  \draw (5.4, 0) node {=};

  \draw (6,0) .. controls (6,1) and (7.3, 1) .. (7.8, 0.7)
  .. controls (8.3, 0.6) .. (8.8, 0.7)
  .. controls (9.3, 1) and (10.6, 1) .. (10.8, 0.4);
  \draw (10.8, 0.4) .. controls (10.9, 0.2) .. (11, 0.2);
  \draw (11, -0.2) .. controls (10.9, -0.2) .. (10.8, -0.4);
  \draw (10.8, -0.4) .. controls (10.6, -1) and (9.3, -1) .. (8.8, -0.7)
  .. controls (8.3, -0.6) .. (7.8, -0.7)
  .. controls (7.3, -1) and (6, -1) .. (6, 0);
  \draw (6.5, 0) .. controls (6.5, -0.25) and (7.7, -0.25) .. (7.7, 0);
  \draw (6.7, -0.15) .. controls (6.7, 0) and (7.5, 0) .. (7.5, -0.15);
  \draw (9, 0) .. controls (9, -0.25) and (10.2, -0.25) .. (10.2, 0);
  \draw (9.2, -0.15) .. controls (9.2, 0) and (10, 0) .. (10, -0.15);
  \draw (11, 0) ellipse (0.1 and 0.2);
  \draw (11, -0.2) arc (-90:90:0.2);
  \draw (8.3, -1.2) node {$M'$};
\end{tikzpicture}
\end{figure}

\end{itemize}

\section{Computations of the symplectic forms}

The aim of this section is to compute the 1-shifted symplectic form on $[G/G]$ via the AKSZ construction and compare it to the form $\omega_0 + u\omega_1$ defined previously.

\subsection{Differential forms on quotient stacks}

Let us describe the isomorphism \[\Gamma(\B G, \Omega^n)\rightarrow\Tot\Gamma(\B_\bullet G, \Omega^n).\] This can be thought of as the Dolbeault version of the Chern-Weil homomorphism \[\Sym^n(\mathfrak{g}^*)^G\rightarrow \mH^{n,n}(\B_\bullet G).\]

This weak equivalence can be found e.g. in \cite{B}, where it is written as a zig-zag of quasi-isomorphisms; we will need a more explicit single quasi-isomorphism.

More generally, consider a smooth scheme $X$ with a $G$-action. Our goal is to write down the descent quasi-isomorphism
\[\Gamma([X/G], \Omega^n)\rightarrow \Gamma(X/G, \Omega^n)\]
for low degrees $n$.

As $\bL_{[X/G]}\cong (\bL_X\rightarrow \mathfrak{g}^*\otimes \cO_X)$, we can identify
\[\Gamma([X/G], \Omega^n)\cong \Sym^n(\Omega^1(X)[1]\oplus \mathfrak{g}^*\otimes \cO(X))^G[-n].\]

\subsubsection{} Let us begin with $n=0$. Then we want to see that the inclusion
\[
\xymatrix{
0 \ar[r] & \cO(X)^G \ar[r] \ar[d] & 0 \ar[r] \ar[d] & 0 \\
0 \ar[r] & \cO(X) \ar[r] & \cO(X\times G) \ar[r] & ...
}
\]
is a quasi-isomorphism. This immediately follows from the fact that the functor of $G$-invariants is exact for a reductive group.

\subsubsection{} Let's move on to the case $n=1$. Then we need to produce a quasi-isomorphism
\[
\xymatrix{
0 \ar[r] & \Omega^1(X)^G \ar[r] \ar[d] & (\cO(X)\otimes \mathfrak{g}^*)^G \ar[r] \ar[d] & 0\ar[r] \ar[d] & 0\\
0 \ar[r] & \Omega^1(X) \ar[r] & \Omega^1(X\times G) \ar[r] & \Omega^1(X\times G\times G) \ar[r] & ...
}
\]

As before, we have a quasi-isomorphism
\[
\xymatrix{
0 \ar[r] & \Omega^1(X)^G \ar[r] \ar[d] & 0 \ar[r] \ar[d] & 0 \\
0 \ar[r] & \Omega^1(X) \ar[r] & \Omega^1(X)\otimes\cO(G) \ar[r] & ...
}
\]

Similarly, we have a quasi-isomorphism
\[
\xymatrix{
0 \ar[r] & (\cO(X)\otimes\mathfrak{g}^*)^G \ar[r] \ar[d] & 0 \ar[r] \ar[d] & 0 \\
0 \ar[r] & \cO(X)\otimes \mathfrak{g}^* \ar[r] & \cO(X)\otimes\cO(G)\otimes\mathfrak{g}^* \ar[r] & ...
}
\]

Note, that $G$ acts on both $X$ and $\mathfrak{g}^*$.

Finally, we can identify $\mathfrak{g}^*$ with right-invariant one-forms $\Omega^1(G)^G$ via the map
\[\mathfrak{g}^*\rightarrow \Omega^1(G)^G\]
sending
\[\phi\mapsto \phi(\overline{\theta}).\]

Expanding the $G$-invariants, we arrive at a complex whose $n$-th term is
\[\Omega^1(X)\otimes \cO(G)^{\otimes n} \oplus \bigoplus_k \cO(X)\otimes \cO(G)\otimes ...\otimes \Omega^1(G)\otimes \cO(G)\otimes ...,\]
where in the second sum the $k$-th term in the tensor product is $\Omega^1(G)$. This space is isomorphic to $\Omega^1(X\times G^{\times n})$ and the complex is exactly the \v{C}ech complex $\Omega^1(X/G)$.

Therefore, the map $\Gamma([X/G], \Omega^1)\rightarrow \Gamma(X/G, \Omega^1)$ has the following components:
\begin{itemize}
\item The map $\Omega^1(X)^G\rightarrow \Omega^1(X)$ is the standard inclusion,

\item The map $(\cO(X)\otimes \mathfrak{g}^*)^G\rightarrow \Omega^1(X\times G)$ is $t\mapsto -t(p_2^*\overline{\theta})$, where $p_2\colon X\times G\rightarrow G$ is the projection.
\end{itemize}

\subsubsection{} Finally, let's discuss the case $n=2$. We have to give a quasi-isomorphism
\[
\xymatrix{
0 \ar[r] & \Omega^2(X)^G \ar[r] \ar[d] & (\Omega^1(X)\otimes \mathfrak{g}^*)^G \ar[r] \ar[d] & (\cO(X)\otimes \Sym^2(\mathfrak{g}^*))^G\ar[r] \ar[d] & 0\\
0 \ar[r] & \Omega^2(X) \ar[r] & \Omega^2(X\times G) \ar[r] & \Omega^2(X\times G\times G) \ar[r] & ...
}
\]

Let's first include the complex
\[0\rightarrow \Omega^2(X)^G\rightarrow (\Omega^1(X)\otimes \mathfrak{g}^*)^G\rightarrow (\cO(X)\otimes \Sym^2(\mathfrak{g}^*))^G\rightarrow 0\]
into the complex
\[
\xymatrix@R-2pc{
0 \ar[r] & \Omega^2(X)^G \ar[r] \ar[ddr] & (\Omega^1(X)\otimes \mathfrak{g}^*)^G \ar[r] & (\cO(X)\otimes \Sym^2(\mathfrak{g}^*))^G \\
& & \oplus & \oplus \\
& & (\cO(X)\otimes \wedge^2(\mathfrak{g}^*))^G \ar^{\id}[r] & (\cO(X)\otimes \wedge^2(\mathfrak{g}^*))^G
}
\]

Here the map $\Omega^2(X)^G\rightarrow (\cO(X)\otimes\wedge^2(\mathfrak{g}^*))^G$ is given by the contraction of a two-form on $X$ with two vector fields generating the $G$-action. It is immediate that the inclusion is a quasi-isomorphism. This follows from the fact that an element $\omega\in\Omega^2(X)^G$ is closed in the old complex if $\iota_{a(v)}\omega = 0$ for $a(v)$ the action vector field of $v\in\mathfrak{g}$, while it is closed in the new complex if $\iota_{a(v)}\omega = 0$ and $\iota_{a(v)}\iota_{a(w)}\omega = 0$. Clearly, these two conditions are equivalent.

Let us combine $\Sym^2(\mathfrak{g}^*)\oplus \wedge^2(\mathfrak{g}^*)\cong \mathfrak{g}^*\otimes \mathfrak{g}^*$. As before, we can expand the $G$-invariants in all terms to obtain a complex of the form
\[
\xymatrix@R-2pc{
0 \ar[r] & \Omega^2(X)\otimes \cO(G^{\times \bullet}) \ar[r] \ar[ddr] & \Omega^1(X)\otimes \cO(G^{\times \bullet})\otimes\mathfrak{g}^* \ar[r] & \cO(X)\otimes \cO(G^{\times \bullet})\otimes \mathfrak{g}^*\otimes\mathfrak{g}^* \\
& & \oplus &  \\
& & \cO(X)\otimes \cO(G^{\times \bullet})\otimes \wedge^2(\mathfrak{g}^*) \ar[uur] & 
}
\]

As before, we have identifications $\Omega^n(G)^G\cong \wedge^n(\mathfrak{g}^*)$ as $G$-representations with respect to the left $G$-action. Moreover, we now have three $G$-actions on $\Omega^2(G\times G)$: the left, right and middle actions. The subspace of two-forms on $G\times G$ invariant with respect to the middle and right $G$-actions is isomorphic to $\wedge^2(\mathfrak{g}^*\oplus \mathfrak{g}^*)$. If we further restrict to two-forms being a product of two one-forms on the two $G$-factors, we get the $G$-representation $\mathfrak{g}^*$ that appears in our complex. Finally, expanding all these $G$-invariants we recover the \v{C}ech complex $\Omega^2(X/G)$.

Let \[\sum_i e_i\otimes e^i\in \mathfrak{g}^*\otimes \mathfrak{g}\]
be the canonical element which is the image of the identity operator under $\End(\mathfrak{g})\cong \mathfrak{g}^*\otimes \mathfrak{g}$.

The map $\Gamma([X/G], \Omega^2)\rightarrow \Gamma(X/G, \Omega^2)$ has the following components:
\begin{itemize}
\item The map $\Omega^2(X)^G\rightarrow \Omega^2(X)$ is the standard inclusion,

\item The map $(\mathfrak{g}^*\otimes_k \Omega^1(X))^G\rightarrow \Omega^2(X\times G)$ is given by
\begin{equation}
t\mapsto \sum_i t(e_i)\wedge e^i(p_2^*\overline{\theta}) - \frac{1}{2}\sum_{i,j}\iota_{a(e_j)} t(e_i) e^i(p_2^*\overline{\theta})\wedge e^j(p_2^*\overline{\theta}),
\label{twoformcech}
\end{equation}
where $t\in(\mathfrak{g}^*\otimes_k \Omega^1(X))^G$,

\item The map $(\Sym^2(\mathfrak{g}^*)\otimes_k \cO(X))^G\rightarrow \Omega^2(X\times G\times G)$ is
\[(-,-)\mapsto -\frac{1}{2}(p_2^*\overline{\theta}, \Ad_{g_1}p_3^*\overline{\theta}),\]
where $(-,-)$ is the symmetric bilinear form associated to the quadratic form on Lie algebra \[q(-)\in \Sym^2(\mathfrak{g}^*)\otimes_k \cO(X);\] in particular, we have $(v, v) = 2q(v)$ for $v\in\mathfrak{g}$.
\end{itemize}

\subsubsection{}
Let us apply the considerations above to the case $X = \pt$. Then the bilinear form $(-,-)\in\Sym^2(\mathfrak{g}^*)^G$ is sent to 
\[\phi=-\frac{1}{2}(p_1^*\theta, p_2^*\overline{\theta}) \in\Omega^2(G\times G).\]

We have a diagram
\[
\xymatrix{
\Omega^{2, cl}(\B G) \ar^{\sim}[r] \ar_{\sim}[d] & \Omega^{2, cl}(\B_\bullet G) \ar[d] \\
\Omega^2(\B G) \ar^{\sim}[r] & \Omega^2(\B_\bullet G)
}
\]

Therefore, the map $\Omega^{2, cl}(\B_\bullet G)\rightarrow \Omega^2(\B_\bullet G)$ is a quasi-isomorphism. In other words, the space of keys of a two-form is contractible.

In \autoref{lemma:chimult} we have shown that $\phi - u\omega_1\in\Omega^{2, cl}(\B_\bullet G)$ which we take to be the image of \[q\in \Sym^2(\mathfrak{g}^*)^G\cong \Omega^{2, cl}(\B G)\] in $\Omega^{2, cl}(\B_\bullet G)$.

\subsection{Symplectic structure on $[G/G]$}

Given the explicit description of the isomorphism $\Gamma(\B G, \Omega^2)\cong \Tot\Gamma(\B_\bullet G, \Omega^2)$, let us now compute the integral transform of the symplectic structure on $\B G$ along
\[
\xymatrix{
& [G/G]\times S^1 \ar_{p}[dl] \ar^{\ev}[dr] & \\
[G/G] && \B G.
}
\]

We can think of the map $\ev$ as a self-homotopy $h$ of the map $[G/G]\rightarrow \B G$ classifying the $G$-torsor $G\rightarrow [G/G]$. The self-homotopy $h$ induces a chain map $\overline{h}\colon \Omega^2(\B G, 2)\rightarrow \Omega^2([G/G], 1)$, which coincides with $p_* \ev^*$. Although we could obtain the answer in this way, we will give a more straightforward computation of $p_* \ev^*$ applicable in more general situations. We will use simplicial techniques, a good introduction is \cite{GJ}.

We use the smallest model of the simplicial circle $S^1$ which is generated by one 0-simplex $q$ and one 1-simplex $\tau$ \cite{Lo}. The simplicial set $S^1_\bullet$ has the following components in low degrees:
\[S^1_0 = \{q\}, \quad S^1_1 = \{s_0 q, \tau\}, \quad S^1_2 = \{s_0^2 q, s_1\tau, s_0\tau\}, ...\]

We will also need the simplices of $\Delta^1$:
\[\Delta^1_0 = \{p_0, p_1\}, \quad \Delta^1_1 = \{s_0 p_0, s_0p_1, I\}, \quad \Delta^1_2 = \{s_0^2 p_0, s_0^2 p_1, s_1I, s_0 I\}, ...\]

Let's begin by describing the isomorphism $\Map_{\dSt}(S^1_{\B}, \B G)\cong [G/G]$. Since $S^1$ is a finite simplicial set, this isomorphism can be obtained by a sheafification of the isomorphism $\Map(S^1_\bullet, \B_\bullet G)\cong G/G$ of simplicial schemes. We will only need the explicit map on 0- and 1-simplices.

First, $\Hom(S^1_\bullet, \B_\bullet G)\cong G$ since a map $S^1_\bullet\rightarrow \B_\bullet G$ is necessarily trivial on 0-simplices and is uniquely determined on $\tau$, the only other nondegenerate simplex.

Next, the isomorphism $\Hom(S^1_\bullet\times \Delta^1, \B_\bullet G)$ is seen in the following way. A map \[f\colon S^1_\bullet\times\Delta^1\rightarrow \B_\bullet G\] is again trivial on 0-simplices and is uniquely determined by the values on nondegenerate 1-simplices subject to the associativity conditions coming from the face maps on 2-simplices. The nondegenerate 1-simplices in $S^1_\bullet\times \Delta^1_\bullet$ are
\[\{s_0q\}\times\{I\}, \quad \{\tau\}\times \{s_0p_0\}, \quad \{\tau\}\times \{s_0p_1\}, \quad \{\tau\}\times \{I\}.\]

We denote the corresponding values of $G$ by
\begin{align*}
g_{s_0q, I}&=f(\{s_0q\}\times\{I\}), \quad g_{\tau, s_0p_0}=f(\{\tau\}\times \{s_0p_0\},)\\
g_{\tau, s_0p_1}&=f(\{\tau\}\times \{s_0p_1\}), \quad g_{\tau, I}=f(\{\tau\}\times \{I\}).
\end{align*}

For any two-simplex $y\in \B_2 G$ we have $d_2(y) d_0(y) = d_1(y)$. Applying it to $f(\{s_1\tau\}\times \{s_0 I\})$ and $f(\{s_0\tau\}\times \{s_1 I\})$ we get the relations
\[g_{\tau, I} = g_{\tau, s_0p_0} g_{s_0 q, I},\quad g_{\tau, I} = g_{s_0q, I} g_{\tau, s_0p_1}.\]

Therefore, $\Hom(S^1_\bullet\times \Delta^1_\bullet, \B_\bullet G)\cong G\times G$ via
\[f\mapsto (f(\{\tau\}\times \{s_0p_0\}), f(\{s_0 q\}\times \{I\})).\]

The face maps $G\times G\rightrightarrows G$ are
\[d_0(g, h) = h^{-1} gh, \quad d_1(g, h) = g.\]

Once we know the isomorphism $\Map(S^1_\bullet, \B_\bullet G)\cong G/G$ we can easily write down the evaluation map \[\ev\colon G/G\times S^1_\bullet\rightarrow \B_\bullet G.\]

The maps \[(G/G)_n\times \{s_0^n q\}\rightarrow G^{\times n}\] are simply projections on the last $n$ components of $(G/G)_n=G^{\times(n+1)}$. The map \[G\times G\times \{\tau\}\rightarrow G\] is $(g, h)\mapsto gh$ (this is $g_{\tau, I}$ in the previous notation). The map on 2-simplices \[G\times G\times G\times S^1_2\rightarrow G\times G\] is then uniquely determined from the simplicial identities. The map
\[\ev_{s_0\tau}\colon G\times G\times G\times \{s_0\tau\}\rightarrow G\times G\]
is
\[(g, h_1, h_2)\mapsto (h_1, h_1^{-1}gh_1h_2).\]
The map
\[\ev_{s_1\tau}\colon G\times G\times G\times \{s_1\tau\}\rightarrow G\times G\]
is
\[(g, h_1, h_2)\mapsto (gh_1, h_2).\]

The pullbacks of $\phi$ are
\begin{align*}
\ev_{s_0\tau}^*\phi &= -\frac{1}{2}(h_1^*\theta, \Ad_{h_1} g^*\overline{\theta}) - \frac{1}{2}(h_1^*\theta, \Ad_{g^{-1}h_1} h_1^*\overline{\theta}) - \frac{1}{2}(h_1^*\theta, \Ad_{h_1^{-1}g^{-1}h_1} h_2^*\overline{\theta})\\
&= -\frac{1}{2}(h_1^*\overline{\theta}, g^*\overline{\theta}) - \frac{1}{2}(h_1^*\overline{\theta}, \Ad_{g^{-1}} h_1^*\overline{\theta}) - \frac{1}{2}(h_1^*\overline{\theta}, \Ad_{h_1^{-1}g^{-1}} h_2^*\overline{\theta}) \\
\ev_{s_1\tau}^*\phi &= -\frac{1}{2}(\Ad_{h_1} g^*\theta, h_2^*\overline{\theta}) - \frac{1}{2}(h_1^*\overline{\theta}, h_2^*\overline{\theta}).
\end{align*}

Finally, the pushforward $p_*\ev^*\phi$ can be computed by composing the Eilenberg--MacLane map \cite[Definition 29.7]{May} with the integral along $S^1$. The result is
\[p_*\ev^*\phi = -s_1^* \ev^*_{s_0\tau}\phi + s_0^* \ev^*_{s_1\tau}\phi,\]
where $s_0$ and $s_1$ are the degeneracy maps $G\times G\rightrightarrows G\times G\times G$.

We obtain
\[p_*\ev^*\phi = -\frac{1}{2}(g^*\theta, h^*\overline{\theta}) + \frac{1}{2}(h^*\overline{\theta}, g^*\overline{\theta}) + \frac{1}{2}(h^*\overline{\theta}, \Ad_{g^{-1}} h^*\overline{\theta}).\]

So far we have computed the image of $\phi$ in $\Omega^2(G/G, 1)$ using the \v{C}ech presentation. Now we show that the \v{C}ech cocycle comes from the degree 1 two-form $\omega_0$ under the map \[\Omega^2([G/G], 1)\rightarrow \Omega^2(G/G, 1).\] Indeed, the image of $\omega_0$ under \eqref{twoformcech} is
\begin{align*}
&-\frac{1}{2}\sum_i(g^*\theta + g^*\overline{\theta}, e_i)\wedge h^*e^i(\overline{\theta}) + \frac{1}{4}\sum_{i,j}\iota_{e_j^R-e_j^L}(\theta+\overline{\theta}, e_i)\cdot e^i(h^*\overline{\theta})\wedge e^j(h^*\overline{\theta})\\
=&-\frac{1}{2}(g^*\theta + g^*\overline{\theta}, h^*\overline{\theta}) + \frac{1}{4}\sum_j (\Ad_{g^{-1}}(e_j) - \Ad_g(e_j), h_2^*\overline{\theta})\wedge e^j(h^*\overline{\theta})\\
=&-\frac{1}{2}(g^*\theta + g^*\overline{\theta}, h^*\overline{\theta}) + \frac{1}{2}(\Ad_gh^*\overline{\theta}, h^*\overline{\theta}).
\end{align*}

All the calculations in this section are summarized in the following theorem.

\begin{thm}
The integral transform of the quadratic form $q\in \Sym^2(\mathfrak{g}^*)^G\cong \Omega^2(\B G, 2)$ under
\[
\xymatrix{
& [G/G]\times S^1 \ar_{p}[dl] \ar^{\ev}[dr] & \\
[G/G] && \B G
}
\]
is equal to \[\omega_0 = -\frac{1}{2}(\theta + \overline{\theta}, -)\in \Omega^2([G/G], 1).\]
\end{thm}

Under the transgression map \[p_*\ev^*\colon \Omega^{2, cl}(\B_\bullet G, 2)\rightarrow \Omega^{2, cl}(G/G, 1)\] the form $\phi - u\omega_1$ is sent to $p_*\ev^*\phi + u\omega_1$. This coincides with the image of $\omega_0 + u\omega_1$ under the map $\Omega^{2, cl}([G/G], 1)\rightarrow \Omega^{2, cl}(G/G, 1)$.

\begin{thm}
The integral transform of the quadratic form $q\in \Sym^2(\mathfrak{g}^*)^G\cong \Omega^{2,cl}(\B G, 2)$ under
\[
\xymatrix{
& [G/G]\times S^1 \ar_{p}[dl] \ar^{\ev}[dr] & \\
[G/G] && \B G
}
\]
is equal to \[\omega_0+u\omega_1=-\frac{1}{2}(\theta+\overline{\theta},-) + \frac{u}{12}(\theta, [\theta, \theta])\in \Omega^{2,cl}([G/G], 1).\]
\end{thm}

\subsection{Disk}

Consider the restriction map $\Map_{\dSt}(D_{\B}, \B G)\rightarrow \Map_{\dSt}(S^1_{\B}, \B G)$. It gives the map \[i\colon [\pt/G]\rightarrow [G/G]\] which is simply the inclusion of the identity element. The pullback of $\omega_0 + u\omega_1$ to $[\pt/G]$ is zero, which is the required Lagrangian structure: the space of isotropic structures on $i$ is $\Omega^{2, cl}(\B G, 0)$, which is contractible.

Explicitly, to check that $i$ is Lagrangian, we have to prove exactness of the sequence of $G$-representations
\[\mathfrak{g}[1]\rightarrow \mathfrak{g}[1]\oplus \mathfrak{g}\rightarrow \mathfrak{g}^*.\]
Indeed, the first map is the obvious inclusion and the second map is the nondegenerate pairing on $\mathfrak{g}$ applied to the second summand.

More generally, consider a conjugacy class $C\subset G$ of an element $g\in G$. Suppose $H$ is the stabilizer of $g$ under the adjoint action. Then $[C/G]\cong [\pt/H]$. The inclusion map $i\colon [\pt/H]\rightarrow [G/G]$ is again isotropic since $\Omega^{2, cl}(\B H, 0)$ is contractible ($H$ is reductive). The tangent complex $\bT_{[G/G]}$ restricted to $g\in G$ is
\[\mathfrak{g}\rightarrow \mathfrak{g}\] in degrees $-1$ and $0$ with the differential is $x\mapsto x-\Ad_g(x)$. We have to prove exactness of the sequence
\[\mathfrak{h}[1]\rightarrow \bT_{[G/G]}|_g\rightarrow \mathfrak{h}^*.\]

The fiber of the second map is the complex
\[\mathfrak{g}\rightarrow \mathfrak{g}\rightarrow \mathfrak{h}^*,\]
in degrees $[-1, 1]$ with the first differential as before and the second differential given by the composite $\mathfrak{g}\cong\mathfrak{g}^*\twoheadrightarrow \mathfrak{h}^*$. To show that the sequence written before is exact, we just have to prove that the cohomology of this complex is $\mathfrak{h}$ concentrated in degree $-1$. Indeed, $H^1 = 0$ as $\mathfrak{g}^*\rightarrow \mathfrak{h}^*$ is surjective. 

The closed elements in degree 0 are elements in $\mathfrak{h}^\perp\subset \mathfrak{g}$. Since \[\mathfrak{h}\cong \Ker(1-\Ad_{g^{-1}}),\] we have \[\mathfrak{h}^\perp\cong \im(1-\Ad_g).\] Therefore, every closed element in degree 0 is exact. Finally, closed elements in degree $-1$ are exactly elements in $\mathfrak{h}\subset \mathfrak{g}$. This proves the exactness of the sequence.

\subsection{Fusion}

Let us compute the Lagrangian structure on the correspondence
\[
\xymatrix{
& [(G\times G)/G] \ar[dl] \ar[dr] &\\
[G/G] & & [G/G]\times [G/G]
}
\]
obtained in the Chern-Simons theory by considering the pair of pants $M$ as a cobordism from $S^1\sqcup S^1$ to $S^1$.

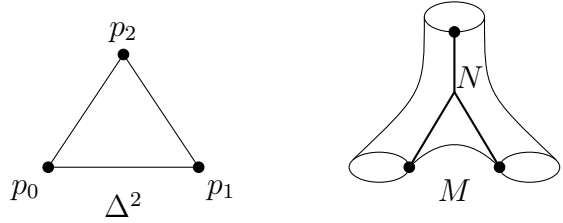
\begin{figure}[h]
\centering
\begin{tikzpicture}
  \draw (0, 0) -- (1, 1.5);
  \draw (1, 1.5) -- (2, 0);
  \draw (0, 0) -- (2, 0);
  \filldraw (0, 0) circle (2pt);
  \filldraw (1, 1.5) circle (2pt);
  \filldraw (2, 0) circle (2pt);
  \draw (1, -0.5) node {$\Delta^2$};
  \draw (-0.3, -0.3) node {$p_0$};
  \draw (2.3, -0.3) node {$p_1$};
  \draw (1, 1.8) node {$p_2$};

  \draw (4.05,0.1) .. controls (5, 1) .. (5, 2);
  \draw (4.4,0) ellipse (0.4 and 0.2);
  \draw (5.4,2) ellipse (0.4 and 0.2);
  \draw (5.8, 2) .. controls (5.8, 1) .. (6.75, 0.1);
  \draw (6.4, 0) ellipse (0.4 and 0.2);
  \draw (4.8, 0) .. controls (5.2, 0.4) and (5.6, 0.4) .. (6.0, 0);
  \draw [thick] (4.8, 0) -- (5.4, 1);
  \draw [thick] (6, 0) -- (5.4, 1);
  \draw [thick] (5.4, 1) -- (5.4, 1.8);
  \filldraw (4.8, 0) circle (2pt);
  \filldraw (6, 0) circle (2pt);
  \filldraw (5.4, 1.8) circle (2pt);
  \draw (5.4, -0.3) node {$M$};
  \draw (5.6, 1.2) node {$N$};
\end{tikzpicture}
\caption{A 2-simplex and a pair of pants}
\end{figure}

First, we have to find a convenient simplicial set representing $M$. Observe that we can contract a subset $N$ drawn in the picture to obtain a topological space homeomorphic to a 2-simplex modulo its vertices. In other words, \[\Delta^2/(p_0\sqcup p_1\sqcup p_2)\cong M/N.\]

Therefore, we can construct $M$ as a simplicial set with one nondegenerate 0-simplex, three nondegenerate 1-simplices and one nondegenerate 2-simplex. Its simplices are
\[M_0 = \{q\},\quad M_1 = \{s_0q, \tau_0, \tau_1, \tau_2\}, \quad M_2 = \{s_0^2q, s_1 \tau_0, s_0\tau_0, s_1\tau_1, s_0\tau_1, s_1\tau_2, s_0\tau_2, M\}, ...\]

The boundary maps are $d_0(M) = \tau_0$, $d_1(M) = \tau_1$ and $d_2(M) = \tau_2$.

Let's construct an isomorphism $\Map(M_\bullet, \B_\bullet G)\cong (G\times G)/G$. Indeed, $\Hom(M_\bullet, \B_\bullet G)\cong G\times G$ as any map $f\colon M_\bullet\rightarrow \B_\bullet G$ is uniquely determined by its values $g_1 = f(\tau_2)$ and $g_2 = f(\tau_0)$. Then, for instance, $f(\tau_1) = g_1g_2$ and $f(M) = (g_1, g_2)$.

One can check that $\Hom(M_\bullet\times \Delta^1_\bullet, \B_\bullet G)\cong G\times G\times G$. The map is
\[f\mapsto (f(\{\tau_2\}\times \{s_0p_0\}), f(\{\tau_0\}\times \{s_0p_0\}), f(\{s_0q\}\times \{I\})).\]

From these considerations we can easily compute the evaluation map
\[\ev\colon(G\times G)/G\times M_\bullet\rightarrow \B_\bullet G.\]

For instance,
\[\ev_2\colon G\times G\times G\times G\times \{M\}\rightarrow G\times G\]
is
\[\ev_2(g_1, g_2, h_1, h_2) = (g_1h_1, h_1^{-1}g_2h_1h_2).\]

Using the Eilenberg--MacLane map the transgression is
\[p_*\ev^* \phi = s_0^*s_1^* \phi = \phi.\]

\begin{thm}
$(m, p_1, p_2)\colon [G/G]\leftarrow [(G\times G)/G]\rightarrow [G/G]\times [G/G]$ is a Lagrangian correspondence with the isotropic structure $\phi\in\Omega^{2, cl}(G\times G)$.
\end{thm}

Suppose $\mu=(\mu_1,\mu_2)\colon M\rightarrow G\times G$ is a $G\times G$-equivariant map. Then $\tilde{\mu}\colon M\rightarrow G\times G\stackrel{m}\rightarrow G$ is $G$-equivariant for the diagonal action of $G$.

If $[M/(G\times G)]\rightarrow [G/G]\times [G/G]$ is Lagrangian, we have a section $\omega\in \mH^0(M, \mu^*(\cT\boxtimes \cT))$. Using the multiplicative structure on $\cT$ we get a section $\tilde{\omega}\in \mH^0(M, \mu^*m^*\cT)$; in fact, since $\cT$ is trivial, we can write it as
\[\tilde{\omega} = \omega - \mu^*\phi = \omega + \frac{1}{2}(\mu_1^*\theta, \mu_2^*\overline{\theta}).\]

We see that $[M/G]$ with the moment map coming from the product $\mu_1\mu_2$ is the internal fusion of $[M/(G\times G)]$.

\subsection{Punctured torus}

\subsubsection{}
If
\[1\rightarrow H\rightarrow\tilde{G}\rightarrow G\rightarrow 1\]
is a central extension of $G$ by $H$, we can canonically lift commutators $aba^{-1}b^{-1}$ to the central extension: pick any lifts $\tilde{a},\tilde{b}$ of $a$ and $b$. Then $\tilde{a}\tilde{b}\tilde{a}^{-1}\tilde{b}^{-1}$ is a lift of $aba^{-1}b^{-1}$. It is easy to see that the lift of the commutator does not depend on the individual lifts.

One can formulate the same result in the language of multiplicative torsors. Consider $G\times G$ with the moment map $\mu\colon G\times G\rightarrow G$ given by the commutator. Let us use the notation $f(g)=(g, g^{-1})$ and $\overline{f}(a, b)=(a,b,a^{-1},b^{-1})$. Then the canonical section of $\mu^*\cT$ over $G\times G$ obtained as $(\tilde{a}\tilde{b})(\tilde{a}^{-1}\tilde{b}^{-1})$ is
\[h_0 = -\overline{f}^*m_{12}^*m_{23}^*\phi - \overline{f}^*m_{12}^*p_{23}^*\phi - \overline{f}^*p_{12}^*\phi + p_1^*f^*\phi + p_2^*f^*\phi.\]

For our choice of the multiplicative structure $f^*\phi = 0$. So, we get
\[h_0 = \frac{1}{2}((ab)^*\theta, (a^{-1}b^{-1})^*\overline{\theta}) + \frac{1}{2}((a^{-1})^*\theta, (b^{-1})^*\overline{\theta}) + \frac{1}{2}(a^*\theta, b^*\overline{\theta}).\]

We can simplify it further using $(a^{-1})^*\theta = -a^*\overline{\theta}$; we obtain
\begin{equation}
h_0 = \frac{1}{2}(a^*\theta, b^*\overline{\theta}) + \frac{1}{2}(a^*\overline{\theta}, b^*\theta) + \frac{1}{2}((ab)^*\theta, (a^{-1}b^{-1})^*\overline{\theta}).
\label{doublelagr}
\end{equation}

\subsubsection{}

Let us now compute the Lagrangian structure $h_0$ on the character stack of the punctured torus in the AKSZ formalism.

First, we can represent the double $D(G)$ as a capped pair of pants:

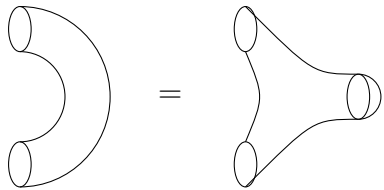
\begin{figure}[h]
\centering
\begin{tikzpicture}
  \draw (0,-1.2) arc (-90:90:1.2);
  \draw (0,0.9) ellipse (0.15 and 0.3);
  \draw (0, -0.9) ellipse (0.15 and 0.3);
  \draw (0,-0.6) arc (-90:90:0.6);
  
  \draw (2,0) node {=};

  \draw (3, 0.9) ellipse (0.15 and 0.3);
  \draw (3, -0.9) ellipse (0.15 and 0.3);
  \draw (3, 1.2) .. controls (3.9, 0.3) .. (4.5, 0.3);
  \draw (3, -1.2) .. controls (3.9, -0.3) .. (4.5, -0.3);
  \draw (4.5, 0) ellipse (0.15 and 0.3);
  \draw (3, 0.6) .. controls (3.25, 0) .. (3, -0.6);
  \draw (4.5, -0.3) arc (-90:90:0.3);
\end{tikzpicture}
\end{figure}

So, we can compute the Lagrangian structure on $D(G)$ by representing it as
\[[G/G]\cong [(G\times G)/G] \times_{[G/G]} [\pt/G].\]

Let $f\colon G\rightarrow G\times G$ be $g\mapsto (g, g^{-1})$, then the $G$-equivariant form on $G$ that equips $[G/G]$ with a Lagrangian structure is
\[f^*\phi = -f^*\frac{1}{2}(p_1^*\theta, p_2^*\overline{\theta}) = \frac{1}{2}(\theta, \theta) = 0.\]

This is not surprising since the double $D(G)$ comes from the cylinder representing the diagonal Lagrangian $[G/G]\rightarrow [G/G]\times [G/G]$, where the Lagrangian structure is trivial.

To compute the Lagrangian structure on the character stack of the punctured torus $[(G\times G)/G]$, let us represent it as a fusion of the double $D(G)$:

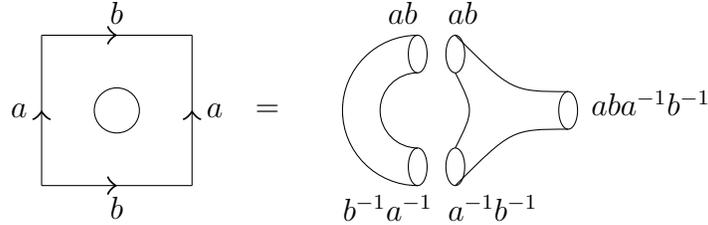
\begin{figure}[h]
\centering
\begin{tikzpicture}
  \begin{scope}[decoration={
    markings,
    mark=at position 0.5 with {\pgftransformscale{2}\arrow{>}}}
    ] 
    \draw[postaction={decorate}] (0,-1)--(2,-1);
    \draw[postaction={decorate}] (2,-1)--(2,1);
    \draw[postaction={decorate}] (0,1)--(2,1);
    \draw[postaction={decorate}] (0,-1)--(0,1);
  \end{scope}
  \draw (1,0) circle (0.3);
  \draw (-0.3,0) node {$a$};
  \draw (2.3,0) node {$a$};
  \draw (1, -1.3) node {$b$};
  \draw (1, 1.3) node {$b$};

  \draw (3, 0) node {=};

  \draw (5, 1) arc (90:270:1);
  \draw (5, 0.75) ellipse (0.125 and 0.25);
  \draw (5, -0.75) ellipse (0.125 and 0.25);
  \draw (5, 0.5) arc (90:270:0.5);
  \draw (4.8, 1.3) node {$ab$};
  \draw (4.6, -1.3) node {$b^{-1}a^{-1}$};

  \draw (5.5, 0.75) ellipse (0.125 and 0.25);
  \draw (5.5, -0.75) ellipse (0.125 and 0.25);
  \draw (5.5, 1) .. controls (6.25, 0.25) .. (7, 0.25);
  \draw (5.5, -1) .. controls (6.25, -0.25) .. (7, -0.25);
  \draw (7, 0) ellipse (0.125 and 0.25);
  \draw (5.5, 0.5) .. controls (5.75, 0) .. (5.5, -0.5);
  \draw (5.6, 1.3) node {$ab$};
  \draw (6, -1.3) node {$a^{-1}b^{-1}$};
  \draw (8.1, 0.1) node {$aba^{-1}b^{-1}$};
\end{tikzpicture}
\end{figure}

This gives a pullback diagram
\[
\xymatrix{
[(G\times G)/G] \ar^{f_1}[r] \ar^{g_1}[d] & [(G\times G)/G] \ar^{g_2}[d] \\
[G/G] \ar^{f_2}[r] & [G/G]\times [G/G],
}
\]
where the maps are
\begin{align*}
f_1(a, b) &= (ab, a^{-1}b^{-1}) \\
g_1(a, b) &= a \\
f_2(a) &= (a, a^{-1}) \\
g_2(a, b) &= (a, b).
\end{align*}

Note, that the diagram has a nontrivial homotopy commutativity data \[h\colon [(G\times G)/G]\times \Delta^1\rightarrow [G/G]\times [G/G]\] given by the path $(ab, b^{-1}a^{-1})\sim (ab, a^{-1}b^{-1})$. On the level of differential forms, $h$ induces a homotopy $\overline{h}\colon  g_1^*f_2^*\Rightarrow f_1^*g_2^*$, i.e. we have
\[\d\overline{h} + \overline{h}\d = g_1^*f_2^* - f_1^* g_2^*.\]

Consider the chain complex $g_1^*\Omega^2_{[G/G]} \oplus f_1^*\Omega^2_{[(G\times G)/G]} \oplus f_1^*g_2^*\Omega^2_{[G/G]\times [G/G]}[1]$ with the differential \[f_1^*g_2^*\Omega^2_{[G/G]\times [G/G]}[1]\rightarrow g_1^*\Omega^2_{[G/G]} \oplus f_1^*\Omega^2_{[(G\times G)/G]}\]
given by
\[\gamma\mapsto (f_2^*\gamma, -g_2^*\gamma).\]

The Lagrangian structure on $[(G\times G)/G]$ is given by the image of $(0, -\phi, p_1^*\omega_0 + p_2^*\omega_0)$ under the map
\[g_1^*\Omega^2_{[G/G]} \oplus f_1^*\Omega^2_{[(G\times G)/G]} \oplus f_1^*g_2^*\Omega^2_{[G/G]\times [G/G]}[1]\rightarrow \Omega^2_{[(G\times G)/G]}\]
given by
\[(\alpha, \beta, \gamma)\mapsto g_1^*\alpha + f_1^*\beta - \overline{h}\gamma.\]

To compute $\overline{h}\colon \Omega^2(G/G\times G/G,1)\rightarrow\Omega^2((G\times G)/G, 0)$, we will use the \v{C}ech presentation of differential forms on $G/G\times G/G$:
\[
\xymatrix{
... \ar@<1ex>[r] \ar[r] \ar@<-1ex>[r] & G\times G\times G \ar@<0.5ex>^-{a}[r] \ar@<-0.5ex>_-{p}[r] & G\times G \ar^{h}[dl] \\
... \ar@<1ex>[r] \ar[r] \ar@<-1ex>[r] & G\times G\times G\times G \ar@<0.5ex>^-{a}[r] \ar@<-0.5ex>_-{p}[r] & G\times G
}
\]

The groupoid maps for $(G\times G)/G$ are
\begin{align*}
a(g, a, b) &= (g^{-1}ag, g^{-1}bg) \\
p(g, a, b) &= (a, b).
\end{align*}

The groupoid maps for $G/G\times G/G$ are
\begin{align*}
a(g_1, g_2, a, b) &= (g_1^{-1}ag_1, g_2^{-1}bg_2) \\
p(g_1, g_2, a, b) &= (a, b).
\end{align*}

The homotopy $h$ is given by
\[h(a, b) = (e, b^{-1}, ab, b^{-1}a^{-1}).\]

The Lagrangian structure on the double $D(G)$ is thus given by the two-form
\[h_0=\frac{1}{2}f_1^*(p_1^*\theta, p_2^*\overline{\theta}) + \frac{1}{2}\overline{h}(p_3^*\theta + p_3^*\overline{\theta} + \Ad_ap_1^*\overline{\theta}, p_1^*\overline{\theta}) + \frac{1}{2}\overline{h}(p_4^*\theta + p_4^*\overline{\theta} - \Ad_bp_2^*\overline{\theta}, p_2^*\overline{\theta}).\]

The second summand is zero since $\overline{h}(p_1^*\overline{\theta}) = 0$. We get
\begin{align*}
h_0 &= \frac{1}{2}((ab)^*\theta, (a^{-1}b^{-1})^*\overline{\theta}) + \frac{1}{2}((ab)^*\theta + (ab)^*\overline{\theta} - \Ad_{b^{-1}a^{-1}} b^*\theta, b^*\theta) \\
&= \frac{1}{2}((ab)^*\theta, (a^{-1}b^{-1})^*\overline{\theta}) + \frac{1}{2}(\Ad_b a^*\theta + b^*\theta + a^*\overline{\theta} + \Ad_{a^{-1}} b^*\overline{\theta} - \Ad_{a^{-1}}b^*\overline{\theta}, b^*\theta) \\
&= \frac{1}{2}((ab)^*\theta, (a^{-1}b^{-1})^*\overline{\theta}) + \frac{1}{2}(a^*\theta, b^*\overline{\theta}) + \frac{1}{2}(a^*\overline{\theta}, b^*\theta),
\end{align*}
which coincides with the previously obtained form $h_0$ \eqref{doublelagr}.

\section{Prequantization}
\label{sect:prequantization}

\subsection{General definition}
\subsubsection{}

Classically, a prequantization of a symplectic manifold $(X,\omega)$ consists of lifting the symplectic form $\omega\in \mH^0(X, \Omega^{2, cl})$ to a line bundle with a connection $L\in \mH^1(X, \cO^\times\rightarrow \Omega^1)$ whose curvature is $\omega$. That is, a prequantization is a lift
\[
\xymatrix{
& X \ar^{\omega}[d] \ar@{-->}_{L}[dl] \\
(\cO^\times\rightarrow\Omega^1) \ar^-{\ddr}[r] & \Omega^{2,cl}
}
\]

\subsubsection{Example}

We will be interested in constructing prequantizations of character stacks, so consider the simplest case of the $\GL_1$ character stack of a torus $\Loc_{\GL_1}(T)$. Removing a disk from the torus and gluing it back in, we obtain a presentation
\begin{align*}
\Loc_{\GL_1}(T) &\cong [(\GL_1\times \GL_1)/\GL_1] \times_{[\GL_1/\GL_1]} [\pt/\GL_1] \\
&\cong \left[((\GL_1\times \GL_1)\times_{\GL_1} \pt)/\GL_1\right]\\
&\cong (\GL_1\times \GL_1) \times (\Omega \GL_1\times \B \GL_1).
\end{align*}

In other words, the character stack $\Loc_{\GL_1}(T)$ is isomorphic to a product of the character \textit{variety} of the torus $\GL_1\times \GL_1$ and the character stack of the sphere \[\Loc_{\GL_1}(S^2) \cong \Omega \GL_1\times \B \GL_1.\] Moreover, the symplectic structure is simply the product symplectic structure.

The symplectic structure on the character variety $\GL_1\times \GL_1$ can be read off from the formula \eqref{doublelagr}. If we denote the coordinates on $\GL_1\times \GL_1$ by $(a,b)$, the symplectic structure is
\[\omega = \ddr\log a\wedge \ddr\log b.\]

Every line bundle on $\GL_1\times \GL_1$ is trivializable, so the curvature of a line bundle with a connection is necessarily exact. But $\omega$ is not exact, so it cannot be prequantized. Alternatively, one can observe that the weight of $\omega$ in the mixed Hodge structure on the character variety is 4, while Chern classes of line bundles have weight 2.

This should be contrasted with the analytic case, where the character variety $\GL_1\times \GL_1$ is isomorphic to the moduli space of holomorphic line bundles with a connection $\Pic^{\flat}(T)$ as a complex manifold once we choose a complex structure on $T$. The space $\Pic^{\flat}(T)$, a twisted cotangent bundle to $\Pic(T)$, admits a prequantization, but the prequantum line bundle is not algebraic when pulled back to $\GL_1\times \GL_1$.

\subsubsection{}

Therefore, we will consider a more general notion of prequantization applicable in the algebraic situation.

An immediate generalization of the notion of prequantization is to a sheaf of complexes $F$ together with a chain map $F\rightarrow \Omega^{2,cl}$. One can also consider a sheaf of infinite loop spaces $F$ together with an $\bE_\infty$-map $F\rightarrow |\Omega^{2, cl}|$.

\begin{defn}
A \textit{prequantization} of an $n$-shifted symplectic stack $(X,\omega)$ is a lift of the symplectic form $\omega\in \Omega^{2, cl}(X, n)$ to a map $\tilde{\omega}\colon X\rightarrow \Omega^n F$.
\end{defn}

Note that we denote by $\Omega$ both the based loop space and the complex of differential forms; we hope the notation will be clear from the context.

\subsection{Prequantization of character stacks}

\subsubsection{Topological field theories}

Let $\Corr$ be the $(\infty, 2)$-category of correspondences of derived stacks with 2-morphisms being correspondences between correspondences. This $(\infty, 2)$-category is defined in \cite{Ha} in the setting of complete $n$-fold Segal spaces where it is denoted by $\mathrm{Span}_2(\dSt)$. Note that every derived stack is fully dualizable with the dual given by the same derived stack. We denote by $\Corr^{\sim}$ the underlying $\infty$-groupoid of invertible morphisms.

Let $\Bordfr_2$ be the $(\infty,2)$-category of framed cobordisms (see \cite{Lu}). We have a functor
\[\Fun(\Bordfr_2, \Corr)\rightarrow \Corr^{\sim}\]
sending a functor $Z\colon \Bordfr_2\rightarrow \Corr$ to its value $Z(\pt)$ on the point. The cobordism hypothesis \cite[Theorem 2.4.6]{Lu} states that it is an equivalence. The inverse functor
\[\Corr^\sim\rightarrow \Fun(\Bordfr_2, \Corr)\] is given by sending \[X\mapsto (M\mapsto \Map_{\dSt}(M_{\B}, X)).\]

We can also consider stacks with closed 2-forms. Let $\CorrO$ be the $(\infty,2)$-category of correspondences of derived stacks equipped with a closed degree $n$ two-form (not necessarily nondegenerate). Any such derived stack is fully dualizable with the dual given by the same derived stack with the opposite two-form. The functor
\[\Fun(\Bordfr_2, \CorrO)\rightarrow \CorrO^{\sim}\]
is again an equivalence. The inverse is given on closed framed $d$-manifolds $M$ by \[Z(M)=\Map(M_\B, X)\] with the two-form given as in \autoref{thm:PTVVAKSZ}.

A conjecture of Lurie and Haugseng \cite[Conjecture 1.4]{Ha} states that the canonical $SO(n)$-action on $\CorrO$ is trivializable. This would imply that given a derived stack $X$ with a closed degree $n$ two-form $\omega$ there is a functor $Z_X\colon \Bordor_2\rightarrow \CorrO$ on the category of \textit{oriented} cobordisms which is $Z_X(M)\cong\Map(M_\B, X)$ forgetting the two-form.

\subsubsection{Algebraic $K$-theory}

Given $R\in\cdga^{\leq 0}$, we have the associated $\K$-theory space of $R$. We denote by $\K$ the sheafification of this space in the \'{e}tale topology.

There is a Chern character map from algebraic $K$-theory to negative cyclic homology whose components we denote by \[\ch_n\colon \K\rightarrow |\Omega^{n, cl}[n]|.\]

Let $\CorrK$ be $(\infty, 2)$-category of correspondences of derived stacks equipped with a map to $\K$. The post-composition with the second Chern character $\ch_2$ gives a functor $\CorrK\rightarrow \CorrOO$.

We have a canonical map $\omega_K\colon \B \GL_n\rightarrow \K$ given by sending a vector bundle to the associated point in the $\K$-theory space. There is a 2-shifted symplectic structure $\omega_{\B \GL_n}$ on $\B \GL_n$ given by the trace pairing. It can be factored as
\[\B \GL_n\stackrel{\omega_K}\rightarrow \K\stackrel{\ch_2}\rightarrow |\Omega^{2, cl}[2]|.\]

In other words, the symplectic structure $\omega_{\B \GL_n}$ can be prequantized to an element $\omega_K$. By the cobordism hypothesis, the restriction to the point
\[\Fun(\Bordfr_2, \CorrK)\rightarrow \CorrK^{\sim}\]
is an equivalence, so $(\B \GL_n, \omega_K)$ defines a functor \[Z_K\colon \Bordfr_2\rightarrow \CorrK,\]
which can be identified as $Z_K(M)\cong \Loc_{\GL_n}(M)$ on the level of underlying derived stacks.

Let $M$ be a closed framed surface (which is necessarily a 2-torus). Then $Z_K(M)\cong \Loc_{\GL_n}(M)$ with a map to $\Omega^2\K$. Its composition with the second Chern character $\ch_2$ gives the symplectic form on $\Loc_{\GL_n}(M)$, thus $Z_K(M)$ is a $\K$-theoretic prequantization of the character stack.

Assuming the Lurie--Haugseng conjecture, we can actually extend this to a $\K$-theoretic prequantization of the character stack $\Loc_{\GL_n}(M)$ of any oriented surface as well. Thus, we obtain a map
\[\Loc_{\GL_n}(M)\rightarrow \Omega^2\K\]
which prequantizes the symplectic structure. We should note, however, that it is not a class in $\K_2(\Loc_{\GL_n}(M))$, the second algebraic $\K$-theory of the character stack, since we took \'{e}tale sheafification of the $\K$-theory space. A related construction by Fock and Goncharov \cite{FG} produces a class in $\Gamma(\Loc_{\GL_n}(M), \cK_2)$, where $\cK_2$ is the Zariski sheafification of the presehaf $\K_2$.

\subsubsection{Beilinson regulator}

The $\K$-theoretic prequantization of character stacks we have constructed, although fairly natural, is not very geometric. Let us show how one can construct a prequantum line bundle on the \textit{complex-analytic} character stack.

Let us remind the basics of Deligne cohomology. Given a complex manifold $X$, we define the complexes of sheaves $\Z_\D(i)$ to be
\[\Z_\D(i) = (\Z\rightarrow \cO\rightarrow \Omega^1\rightarrow ...\rightarrow \Omega^{i-1}).\]

The Deligne cohomology groups $\mH^n(X, \Z_D(i))$ are the hypercohomology groups of these complexes. One can easily see that $\mH^2(X, \Z_D(1))$ parametrizes line bundles on $X$ and $\mH^2(X, \Z_\D(2))$ parametrizes line bundles with connection on $X$.

We have the morphisms $\mH^n(X, \Z_D(i))\rightarrow \Omega^{i, cl}(X, n-i)$ given by the de Rham differential. For instance, the map $\mH^2(X, \Z_D(1))\rightarrow \Omega^{1, cl}(X, 1)$ is the first Chern class of the line bundle and $\mH^2(X, \Z_\D(2))\rightarrow \Omega^{2, cl}(X, 0)$ is the curvature map.

Beilinson \cite{Be} has realized that the Chern character map $\ch\colon \K_i(X)\rightarrow \oplus_n \Omega^{n, cl}(X, n-i)$ can be factored as
\[\K_i(X)\stackrel{\reg}\rightarrow \bigoplus_n \mH^{2n-i}(X, \Z_\D(n))\stackrel{\ddr}\rightarrow \bigoplus_n \Omega^{n, cl}(X, n-i),\]
where the maps $\reg_n\colon \K_i(X)\rightarrow \mH^{2n-i}(X, \Z_\D(n))$ are given by the Beilinson regulator.

Let us now assume that one has defined a category of derived complex-analytic stacks $\dASt_\C$ together with the sheaves $\Omega^{n, cl}$ and $\Z_\D(n)$. Moreover, suppose that we have a universal second Chern character in the Deligne cohomology
\[\B \GL_n^{an}\stackrel{\ch_2^\D}\rightarrow |\Z_\D(2)[4]|\]
such that the composite
\[\B \GL_n^{an}\stackrel{\ch_2^\D}\rightarrow |\Z_\D(2)[4]|\stackrel{\ddr}\rightarrow |\Omega^{2, cl}[2]|\]
is the usual symplectic structure on $\B \GL_n^{an}$.

Repeating the construction of the $\K$-theoretic prequantization, we obtain a class
\[\omega_\D\in \mH^2(\Loc_{\GL_n}(M)^{an}, \Z_D(2)),\]
i.e. a holomorphic line bundle with a connection on the analytic character stack, whose curvature coincides with the canonical symplectic form. In other words, we have constructed a prequantization of the analytic character stack starting from a prequantization of the classifying stack $\B \GL_n$.

\end{document}